\documentclass[]{amsproc}
\usepackage{lineno,hyperref}
\usepackage{graphicx}
\usepackage[dvips]{epsfig}
\usepackage[dvips]{color}
\usepackage{amsmath,amssymb,amscd,amsthm,euscript}
\usepackage{dsfont}
\usepackage{epstopdf}
\usepackage{subfigure}% subcaptions for subfigures
\usepackage{subfigmat}% matrices of similar subfigures, aka small mulitples
\newtheorem{lemma}{Lemma}

\newtheorem{theorem}{Theorem}
\newtheorem{assumption}{Assumption}
\newtheorem{remark}{Remark}
\allowdisplaybreaks \allowdisplaybreaks[4]
\title[Super-convergence on exponential integrator for SHE driven by fBm]{Super-convergence analysis on exponential integrator for stochastic heat equation driven by additive fractional Brownian motion}
\author{Jialin Hong}
\address{Academy of Mathematics and Systems Science, Chinese Academy of Sciences, Beijing 100190, China; School of Mathematical Sciences, University of Chinese Academy of Sciences, Beijing 100049, China}
\curraddr{}
\email{hjl@lsec.cc.ac.cn}
\thanks{}
\author{Chuying Huang}
\address{Academy of Mathematics and Systems Science, Chinese Academy of Sciences, Beijing 100190, China; School of Mathematical Sciences, University of Chinese Academy of Sciences, Beijing 100049, China}
\curraddr{}
\email{huangchuying@lsec.cc.ac.cn (Corresponding author)}
\thanks{Authors are funded by National Natural Science Foundation of China (NO. 11971470 and NO. 11871068).}



\begin{document}
	\begin{abstract}
		In this paper, we consider the strong convergence order of the exponential integrator for the stochastic heat equation driven by an additive fractional Brownian motion with Hurst parameter $H\in(\frac12,1)$. By showing the strong order one of accuracy of the exponential integrator under appropriote assumptions, we present the first super-convergence result in temporal direction on full discretizations for stochastic partial differential equations driven by infinite dimensional fractional Brownian motions with Hurst parameter $H\in(\frac12,1)$. The proof is a combination of Malliavin calculus, the  $L^p(\Omega)$-estimate of the Skorohod integral and the smoothing effect of the Laplacian operator.
	\end{abstract}
\maketitle

\section{Introduction}
The fractional Brownian motion (fBm) with Hurst parameter $H\in(0,1)$ is a family of Gaussian processes, which extends the standard Brownian motion ($H=\frac12$). In particular, if $H\in(\frac12,1)$, the process exhibits long-range dependence properties and the increments are positively correlated. The recent development of the stochastic analysis has revealed that the fBm with Hurst parameter $H\in(\frac12,1)$ provides effective models for describing features of the randomness in various fields, such as hydrology,  telecommunications, traffic networks and financial markets; see e.g. \cite{Duncan09SIMA,FM17,Hu11AOP,Rev68,Nualart07JFA} and references therein. 
These applications motivate numerical researches about stochastic differential equations driven by additive fBms, among which the strong convergence analysis for numerical schemes is an important part.

In general, the strong convergence order of a numerical approximation for a stochastic differential equation is restricted by the regularity of the solution. If the order is consistent with the regularity of the solution of the original equation, then the strong convergence order is called optimal. It is a natural and interesting question whether the order can exceed the regularity. In particular, If the strong convergence order in temporal direction is larger than the exponent of temporal H\"older continuity of the solution, then we say that the numerical scheme is super-convergent in time. In finite dimensional cases, there have been several super-convergence results on numerical schemes for stochastic differential equations driven by additive fBms with Hurst parameter $H\in(\frac12,1)$. For example, the strong order one of accuracy of the Euler scheme is proved for equations with Lipschitz drifts in scalar cases \cite{N06} and in multi-dimensional cases \cite{MC11}, while the exponent of temporal H\"older continuity of the solution is not larger than the Hurst parameter of the fBm. For equations with singular drifts, \cite{hhkw} proves the strong order one of accuracy of the backward Euler scheme and applies the scheme to numerically solve the Cox--Ingersoll--Ross interest model driven by an fBm. To our best knowledge, however, there is no super-convergence result in temporal direction on numerical schemes for stochastic partial differential equations (SPDEs) driven by infinite dimensional fBms.

The goal of this paper is to investigate the super-convergence analysis on the exponential integrator approximating the mild solution of 
the stochastic heat equation (SHE) driven by an infinite dimensional fBm
\begin{align}\label{eq1}
\left\{
\begin{aligned}
dX_t&=\Delta X_t dt+ F(X_t)dt+dW^{{\bf Q}}_t,\quad t\in(0,T],\\
X_0&=u_0\in V.
\end{aligned}
\right.
\end{align}
The infinite dimensional fBm $W^{{\bf Q}}$ is defined by 
\begin{align}\label{WQ}
W^{{\bf Q}}_t:=\sum_{i=1}^{\infty}{\bf Q}^{\frac12} f_i \beta^i_t,\quad t\in[0,T],
\end{align}
where $\{\beta^i\}_{i=1}^{\infty}$ is a sequence of identically distributed and independent scalar fBms with Hurst parameter $H\in(\frac12,1)$, $\{f_i\}_{i=1}^{\infty}$ is an orthonormal basis of another separable Hilbert sapce $U$ and ${\bf Q}\in\mathcal{L}(U,V)$ is a self-adjoint, nonnegative definite and bounded linear operator. Denote by $\{S_t\}_{t\ge0}$ the analytic semigroup generated by $-\Delta$. Then the mild solution reads
\begin{align}\label{mild}
X_t=S_{t}X_0+\int_{0}^{t}S_{t-s}F(X_s)ds+\int_{0}^{t}S_{t-s}dW_s^{{\bf Q}},\quad t\in[0,T],
\end{align} 
where the stochastic integral is defined by the fractional calculus \cite{Duncan09SIMA}. 

In this paper, we focus on the case that 
 $\Delta$ is the Dirichlet Laplacian and $V=L^2(0,1)$ equipped with the inner product $\langle g,\tilde{g}\rangle_V:=\int_{0}^{1}g(x)\tilde{g}(x)dx$. %Indeed, one can also study other types of boundary conditions and regular domains. The techniques presented are also valid.
Then the eigensystem of $A:=-\Delta$ is $\{\lambda_i,e_i\}_{i=1}^{\infty}$ with $\lambda_i=i^2\pi^2$ and $e_i(x)=\sqrt{2}\sin(i\pi x)$, where $\{e_i\}_{i=1}^{\infty}$ forms an orthonormal basis of $V$. 
Defining $U_0:={\bf Q}^{\frac12}U$, we denote by $\dot{V}^\theta$ the domain of $A^{\frac{\theta}{2}}$ endowed with the norm
\begin{align*}
\|x\|_{\dot{V}^\theta}:=\Big\|A^\frac\theta2 x\Big\|_V,\quad x\in \dot{V}^\theta,~\theta\in\mathbb{R}
\end{align*}
and by $(\mathcal{L}^0_2,\langle \cdot,\cdot\rangle_{\mathcal{L}^0_2})$ the space of Hilbert--Schmidt operators  from $U_0$ to $V$ equipped the inner product
\begin{align*}
\langle \Phi_1,\Phi_2\rangle_{\mathcal{L}^0_2}:=\sum_{i=1}^{\infty}\langle \Phi_1{\bf Q}^{\frac12}f_i,\Phi_2{\bf Q}^{\frac12}f_i\rangle_V.
\end{align*}
Our assumptions on \eqref{eq1} are stated as follows.
\begin{assumption}[initial value]\label{x0}
	There exists some $\gamma$ such that 
	\begin{align*}
	u_0\in \dot{V}^{2H+\gamma-1}.
	\end{align*}
\end{assumption}

\begin{assumption}[nonlinear term]\label{F}
The operator $F: V\rightarrow V$ is a Nemytskii operator associated with a function $f\in\mathbb{C}^3(\mathbb{R},\mathbb{R})$ 
such that 
$F(X)(x)=f(X(x))$, $x\in(0,1)$, $X\in V$,
and 
\begin{align*}
\sup_{x\in\mathbb{R}}|f'(x)|+\sup_{x\in\mathbb{R}}|f''(x)|+\sup_{x\in\mathbb{R}}|f'''(x)|<\infty.
\end{align*} 
\end{assumption}

\begin{assumption}[noise term]\label{Q}
	There exists some $\gamma$ such that
	\begin{align*}
	\left\|A^{\frac{\gamma-1}{2}}\right\|_{\mathcal{L}^0_2}<\infty.
	\end{align*}
\end{assumption}

In the following, we formulate our main result for a fully discrete scheme construted by spectral Galerkin method and exponential integrator. 

\begin{theorem}\label{main}
Suppose that $X$ is the mild solution of \eqref{eq1} and that $X^{K,M,N}$ is defined by scheme \eqref{eq4}. Under Assumptions \ref{x0}-\ref{Q} with $\gamma>\max\{3-2H,\frac32\}$, it holds that
\begin{align*}
\sup_{n=0,\cdots,N}\Big\| X_{t_n} -X^{K,M,N}_{t_{n}} \Big\|_{L^2(\Omega;V)}\le CN^{-1}+CM^{-2}+C\big\|{\bf Q}^\frac12 ({\rm Id}_U-\mathcal{P}_K)\big\|_{\mathcal{L}_2(U,\dot{V}^{-2H})},
\end{align*}
where ${\rm Id}_U$ is the identity operator on $U$ and $\mathcal{P}_K$ is the project operator from $U$ onto $U_K:={\rm span}\{f_i:i=1,\cdots,K\}$.
\end{theorem}
Here and in the rest of the paper, we use $C$ as a generic constant which is independent of integers $K,M,N$  in \eqref{eq2}-\eqref{eq4} and may be different from line to line. 
Since the exponent of temporal H\"older continuity of the solution proved in Section \ref{sec3} is the same as the Hurst parameter $H$, Theorem \ref{main} indicates that the exponential integrator is super-convergent in time with strong order one of accuracy. We also remark that the exponential integrator does not require the CFL-type condition appearing in most of explicit methods for SPDEs.

As far as we know, Theorem \ref{main} is the first super-convergence result in  temporal direction on full discretizations for SPDEs driven by infinite dimensional fBms with Hurst parameter $H\in(\frac12,1)$. If the noise is less regular with $\gamma\le \max\{3-2H,\frac32\}$, one can obtain the optimal strong convergence order in temporal direction based on \cite{wang17BIT}. For equations with additive noise which is fractional in space and white in time, we refer to \cite{liu17SINUM,liu18IMA} and references therein for optimal error analysis on numerical approximations. As $H$ tends to $\frac12$, the parameter $\gamma$ goes to $2$ which coincides with the assumption on the SHEs driven by infinite dimensional standard Brownian motions for the strong order one of accuracy of the exponential integrator  \cite{J11SINUM,Kloeden11,Kruse}. 

Compared with the standard Brownian setting, the main diffuculty in the super-convergence analysis on full discretizations for SHEs driven by infinite dimensional fBms lies in that the fBm is neither a Markov process nor a semi-martingale such that the Burkholder--Davis--Gundy inequality is unavailable. As a consequence, we need to take a different strategy to estimate the terms $J_{13}$ and $J_{14}$ originated from the stochastic Taylor's expansion in Lemma \ref{J1}. For term $J_{13}$ which involves a stochastic integral with respect to the fBm, we utilize the Malliavin calculus to sum up the accumulated errors first and then take the expectation, instead of to estimate the strong order of accuracy of the local error first and then do the summation. For term $J_{14}$, to prove the temporal regularity of the mild solution in $L^4(\Omega;V)$,  we combine the $L^p(\Omega)$-estimate of the Skorohod integral with respect to the fBm and the smoothing effect of the Laplacian operator to overcome the difficulty from the dependency of increments of fBm, i.e., to eliminate the influence of the kernel $\phi$ of the covariance of fBm.

The paper is structured as follows. In Section \ref{sec2}, we introduce the Malliavin calculus with respect to the fBm. In Section \ref{sec3}, we show the regularity of the mild solution of \eqref{eq1}. In Section \ref{sec4}, we prove the optimal strong convergence order of the spectral Galerkin method and a priori estimates for the approximate mild solution obtained by the spatial semi-discretization. In Section \ref{sec5}, we establish the super-convergence result in temporal direction on the exponential integrator. Section \ref{sec6} gives a conclusion and future works.

\section{Preliminaries on Malliavin calculus}\label{sec2}
This section introduces the definition of the fBm and  the associated Mallavin calculus. For  more details, we refer to \cite{DT13,Nualart,NS09SPA}.

The $K$-dimensional fBm $\{B^K_t=(\beta^1_t,\cdots,\beta^K_t)\}_{t\in[0,T]}$ with Hurst parameter $H\in(\frac12,1)$ is a centered Gaussian process with continuous sample paths and the covariance
\begin{align*}
\mathbb{E}\big[\beta^i_t\beta^j_s\big]:=\frac12\left(t^{2H}+s^{2H}-|t-s|^{2H}\right)\mathds{1}_{\{i\}}(j)=\left(\int_{0}^{t}\int_{0}^{s}\phi(u,v)dudv\right)\mathds{1}_{\{i\}}(j),
\end{align*}
where
$\phi(u,v):=\alpha_H|u-v|^{2H-2}$, $\alpha_H:=H(2H-1)>0$, $i,j=1\cdots,K$, and $\mathds{1}_{\{i\}}$ is the indicator function.
Define an inner product $\langle\cdot,\cdot\rangle_\mathcal{H}$ by 
\begin{align*}
\langle\left(\mathds{1}_{[0,t_1]},\cdots,\mathds{1}_{[0,t_K]}\right),\left(\mathds{1}_{[0,s_1]},\cdots,\mathds{1}_{[0,s_K]}\right)\rangle_{\mathcal{H}}:=\sum_{i=1}^{K}\mathbb{E}\big[\beta^i_{t_i}\beta^i_{s_i}\big],
\end{align*}
and let the Hilbert space $\left(\mathcal{H},\langle\cdot,\cdot\rangle_{\mathcal{H}}\right)$ be the closure of the space of all $\mathbb{R}^K$-valued step functions on $[0,T]$ with respect to $\langle\cdot,\cdot\rangle_{\mathcal{H}}$. 
Then by extending the mapping
$\left(\mathds{1}_{[0,t_1]},\cdots,\mathds{1}_{[0,t_K]}\right) \mapsto \sum_{i=1}^{K}\beta^i_{t_i}$, we obtain an isometry map $\varphi \mapsto B^K(\varphi)$, which is from $\mathcal{H}$ to the Gaussian space associated with $B^K$.

%Introduce the kernel $K_H(t,s):=c_Hs^{\frac12-H}\int_{s}^{t}(u-s)^{H-\frac32}u^{H-\frac12}du\mathds{1}_{\{s<t\}}$ with $c_H:=\sqrt{\frac{\alpha_H}{\beta(2-2H,H-\frac12)}}$ and $\beta(\cdot,\cdot)$ being the Beta function. Then the covariance has the expression
%\begin{align*}
%\mathbb{E}\Big[\beta^j_t\beta^j_s\Big]=\int_{0}^{t\wedge s}K_H(t,u)K_H(s,u)du.
%\end{align*}
%Define the operator $K^*$ from $\mathcal{H}$ to $L^2([0,T];\mathbb{R}^K)$ by  $$(K^*\varphi)^j_s:=\int_{s}^{T}\varphi^j_t \frac{\partial K_H}{\partial t}(t,s)dt.$$
%Another isometry is then obtained between $\mathcal{H}$ and a closed subspace of $L^2([0,T];\mathbb{R}^K)$ (see \cite{NS09SPA}), i.e., 
%\begin{align}\label{iso}
%<\varphi,\psi>_{\mathcal{H}}=<K^*\varphi,K^*\psi>_{L^2([0,T];\mathbb{R}^m)}.
%\end{align}
%Define the operator $\mathcal{K}$ by $(\mathcal{K}\upsilon)_t:=\int_{0}^{t}K_H(t,s)\upsilon_sds$, for $\upsilon\in L^2([0,T];\mathbb{R}^K)$. The injection from $\mathcal{H}$ to $\mathcal{H}_H:=\mathcal{K}(L^2([0,T];\mathbb{R}^K))$ is denoted by ${\rm R_H}=\mathcal{K}\circ \mathcal{K}^*$, i.e., 
%\begin{align*}
%({\rm R_H}\varphi)_t:=\int_{0}^{t}K_H(t,s)(K^*\varphi)_sds.
%\end{align*}

For the random variable 
\begin{align}\label{Y}
Y=y\big(B^K(\varphi^1),\cdots,B^K(\varphi^M)\big),
\end{align}
where $\varphi^1,\cdots,\varphi^M\in \mathcal{H}$ and $y:\mathbb{R}^M \rightarrow\mathbb{R}$ is bounded with bounded derivatives of any order, the Malliavin derivative of $Y$ is an $\mathcal{H}$-valued random variable defined by 
\begin{align*}
D_tY:=\sum_{i=1}^{M}\frac{\partial y}{\partial x_i}\big(B^K(\varphi^1),\cdots,B^K(\varphi^M)\big)\mathbf \varphi^i_t,\quad t\in [0,T].
\end{align*}
In particular, we denote by $(DY)^{i}$ the Malliavin derivative of $Y$ with respect to $\beta^i$, $i=1,\cdots,K$.
For $p\ge 1$, define $\mathbb{D}^{1,p}$ as the Sobolev space which is the closure of the set containing random variables in the form of \eqref{Y} with the norm 
\begin{align*}
\|Y\|_{\mathbb{D}^{1,p}}:=\Big( \mathbb{E}\big[|Y|^p\big] + \mathbb{E}\big[\|DY\|_{\mathcal{H}}^p\big]  \Big)^{\frac1p}.
\end{align*}
The chain rule holds so that for $\tilde{f}$ with bounded derivative and $Y\in \mathbb{D}^{1,2}$,
\begin{align*}
D\tilde{f}(Y)=\tilde{f}'(Y)DY.
\end{align*}

Let $\delta$ be the adjoint operator of the derivative operator $D$. For an $\mathcal{H}$-valued random variable $\varphi\in L^2(\Omega;\mathcal{H})$, if
\begin{align*}
\big| \mathbb{E}\big[\langle\varphi,DY\rangle_{\mathcal{H}}\big]  \big|\le C(\varphi) \|Y\|_{L^2(\Omega;\mathbb{R})},  \quad \forall~Y\in \mathbb{D}^{1,2},
\end{align*}
we say $\varphi\in {\rm Dom}(\delta)$. Then $\delta(\varphi)\in L^2(\Omega;\mathbb{R})$ is defined by the random variable satisfying
\begin{align}\label{IBP}
\mathbb{E}\big[\langle\varphi,DY\rangle_{\mathcal{H}}\big]=\mathbb{E}\big[Y\delta(\varphi)\big], \quad \forall~Y\in \mathbb{D}^{1,2}.
\end{align}
Indeed, the definition of the Malliavin derivative can be extended to $\mathcal{H}$-valued random variables. Then the space $\mathbb{D}^{1,p}$ can be extended to $\mathbb{D}^{1,p}(\mathcal{H})$ with the norm 
\begin{align*}
\|Z\|_{\mathbb{D}^{1,p}(\mathcal{H})}:=\Big( \mathbb{E}\big[\|Z\|_{\mathcal{H}}^p\big] + \mathbb{E}\big[\|DZ\|_{\mathcal{H}\otimes\mathcal{H}}^p\big]  \Big)^{\frac1p}.
\end{align*}
According to \cite[Proposition 1.3.1]{Nualart}, we have  $\mathbb{D}^{1,2}(\mathcal{H})\subset{\rm Dom}(\delta)$.
Moreover, the Skorohod integral of $\varphi$ with respect to fBm is
\begin{align*}
\int_{0}^{T}\varphi_t\delta B^K_t:=\delta(\varphi),\quad \varphi\in {\rm Dom}(\delta),
\end{align*}
and the integration by parts formula holds that
\begin{align*}
\delta(Y\varphi)=Y\delta(\varphi)-\langle DY,\varphi\rangle_{\mathcal{H}},
\end{align*}
if $\varphi\in {\rm Dom}(\delta)$, $Y\in \mathbb{D}^{1,2}$ and $Y\varphi \in L^2(\Omega;\mathcal{H})$. 

The following lemmas are useful for us to deal with  the stochastic integrals in the regularity analysis and the error estimate.

\begin{lemma}(see also \cite[Lemma 1]{MC11})\label{lm-EIBP}
For $g_1,g_2\in \mathcal{H}$, it holds that 
\begin{align*}
\mathbb{E}\big[Y\delta(g_1)\delta(g_2)\big]=\mathbb{E}\big[\langle D[\langle D[Y],g_2\rangle_{\mathcal{H}}],g_1\rangle_{\mathcal{H}}\big]+\mathbb{E}\big[Y\langle g_1,g_2\rangle_{\mathcal{H}}\big].
\end{align*}
\end{lemma}
\begin{proof}
Using \eqref{IBP} and the chain rule for the Malliavin derivative, we obtain
\begin{align*}
\mathbb{E}\big[Y\delta(g_1)\delta(g_2)\big]&=\mathbb{E}\big[\langle D[Y\delta(g_1)],g_2\rangle_{\mathcal{H}}\big]\\
%&=\mathbb{E}[\langle D[Y]\delta(g_1),g_2\rangle_{\mathcal{H}}]+\mathbb{E}[\langle Yg_1,g_2\rangle_{\mathcal{H}}]\\
&=\mathbb{E}\big[\langle D[Y],g_2\rangle_{\mathcal{H}}\delta(g_1)\big]+\mathbb{E}\big[Y\langle g_1,g_2\rangle_{\mathcal{H}}\big]\\
&=\mathbb{E}\big[\langle D[\langle D[Y],g_2\rangle_{\mathcal{H}}],g_1\rangle_{\mathcal{H}}\big]+\mathbb{E}\big[Y\langle  g_1,g_2\rangle_{\mathcal{H}}\big].
\end{align*}
\end{proof}

\begin{lemma}(\cite[Proposition 1.3.1 and Proposition 1.5.8]{Nualart})\label{Lp}
	Let $p>1$, 
	\begin{align*}
	\| \varphi\|^2_{|\mathcal{H}|}&:=\sum_{i=1}^{K}\int_{[0,T]^2}|\varphi^i_u||\varphi^i_v|\phi(u,v)du dv,\\
	\|D\varphi\|^2_{|\mathcal{H}|\otimes|\mathcal{H}|}&:=\sum_{i,j=1}^{K}\int_{[0,T]^4}\Big|\left(D_{v_1}\varphi^i_{u_1}\right)^{j}\Big|\Big|\left(D_{v_2}\varphi^i_{u_2}\right)^{j}\Big|
	\phi(v_1,v_2)\phi(u_1,u_2)dv_1du_1dv_2du_2. 
	\end{align*}
	If a process $\varphi\in {\rm Dom}(\delta)$ satisfying $ \big\|\mathbb{E}[\varphi]\big\|^p_{|\mathcal{H}|}+ \mathbb{E}\big[
	\|D\varphi\|^p_{|\mathcal{H}|\otimes|\mathcal{H}|}\big]< \infty$, then 
	\begin{align*}
	\mathbb{E}\left[  \left| \int_{0}^{T}\varphi_t \delta B^K_t \right| ^p\right]\le
	C(H,p) \left(   \big\| \mathbb{E}[\varphi]\big\|^p_{|\mathcal{H}|}+ \mathbb{E}\big[
	\|D\varphi\|^p_{|\mathcal{H}|\otimes|\mathcal{H}|}\big]\right).
	\end{align*}
	
	In particular, if $K=1$ and $p=2$, it holds that
	\begin{align*}
	\mathbb{E}\left[  \left| \int_{0}^{T}\varphi_t \delta B^1_t \right| ^2\right]
	=&\mathbb{E}\Bigg[
	\int_{[0,T]^2}\varphi_u\varphi_v\phi(u,v)du dv\\
	&+\int_{[0,T]^4}D_{v_1}\varphi_{u_1}D_{v_2}\varphi_{u_2}
\phi(v_1,u_1)\phi(v_2,u_2)dv_1du_1dv_2du_2\Bigg].
	\end{align*}
\end{lemma}

\section{Well-posedness and regularity}\label{sec3}
In this section, we show the regularity of the mild solution of \eqref{eq1}, which relies on the parameters $\gamma$ and $H$. We begin with several Lemmas concerned about the smoothing effect of the Laplacian operator and the isometry of the stochastic integral with respect to the fBm.

\begin{lemma}(\cite[Lemma B.9]{Kruse})\label{lm-A}
For any $0<t\le T$, $\nu \le 0\le \mu$, $0\le \alpha\le 1$ and $x\in V$, it holds that
\begin{align*}
\|A^\nu \|_{\mathcal{L}(V)}\le C,\qquad
\|&A^{-\alpha} (S_{t}-{\rm Id}_{V}) \|_{\mathcal{L}(V)}\le Ct^\alpha, \\
\|A^\mu S_{t}\|_{\mathcal{L}(V)}\le C t^{-\mu},\qquad
\bigg\|&A^\alpha \int_{s}^{t}S_{t-\sigma}xd\sigma\bigg\|_V\le C|t-s|^{1-\alpha} \|x\|_V.
\end{align*}
\end{lemma}

\begin{lemma}(\cite[Lemma 3.6]{wang17BIT})\label{lm-phi}
	There exists some constant $C=C(H)$ such that for any $0\le \rho\le H$, $0\le s<t\le T$ and $x\in V$, 
	\begin{align*}
	\int_{s}^{t}\int_{s}^{t}\langle A^\rho S_{t-u}x,A^\rho S_{t-v}x\rangle_V \phi(u,v)dudv\le C(t-s)^{2(H-\rho)}\|x\|_V^2.
	\end{align*}
\end{lemma}

\begin{lemma}\label{lm-ito}
	For $\Phi:[0,T]\rightarrow \mathcal{L}^0_2$, it holds that
	\begin{align*}
	\mathbb{E}\left[\left\|\int_{0}^{T}\Phi_tdW^{{\bf Q}}_t\right\|_V^2\right]=\sum_{i=1}^{\infty}\int_{0}^{T}\int_{0}^{T}\langle \Phi_u{\bf Q}^{\frac12}f_{i},\Phi_v{\bf Q}^{\frac12}f_{i}\rangle_V\phi(u,v)dudv,
	\end{align*}
  i.e.,
  	\begin{align*}
  \mathbb{E}\left[\left\|\int_{0}^{T}\Phi_tdW^{{\bf Q}}_t\right\|_V^2\right]=\int_{0}^{T}\int_{0}^{T}\langle \Phi_u,\Phi_v\rangle_{\mathcal{L}^0_2}\phi(u,v)dudv.
  \end{align*}
\end{lemma}
\begin{proof}
	Using the definiton of $W^{{\bf Q}}$ and the Parseval equality, we have
	\begin{align*}
	\mathbb{E}\left[\left\|\int_{0}^{T}\Phi_tdW^{{\bf Q}}_t\right\|_V^2\right]
	%=&\mathbb{E}\left[\sum_{j=1}^{\infty}\langle \int_{0}^{T}\Phi_tdW^{{\bf Q}}_t,e_j\rangle_V^2\right]\\
	%=&\sum_{j=1}^{\infty}\mathbb{E}\left[\langle \sum_{i_1=1}^{\infty}\int_{0}^{T}\Phi_t{\bf Q}^{\frac12}f_{i_1}d\beta^{i_1}_t,e_j\rangle_V\langle \sum_{i_2=1}^{\infty}\int_{0}^{T}\Phi_t{\bf Q}^{\frac12}f_{i_2}d\beta^{i_2}_t,e_j\rangle_V\right]\\
	%=&\sum_{j=1}^{\infty}\sum_{i=1}^{\infty}\mathbb{E}\left[\langle \int_{0}^{T}\Phi_t{\bf Q}^{\frac12}f_{i}d\beta^{i}_t,e_j\rangle_V^2\right]\\
	=\sum_{j=1}^{\infty}\sum_{i=1}^{\infty}\mathbb{E}\left[\left|\int_{0}^{T}\langle \Phi_t{\bf Q}^{\frac12}f_{i},e_j\rangle_Vd\beta^{i}_t\right|^2\right].
	\end{align*}
	Thanks to Lemma \ref{Lp}, we obtain
	\begin{align*}
	&\mathbb{E}\left[\left\|\int_{0}^{T}\Phi_tdW^{{\bf Q}}_t\right\|_V^2\right]\\
	%=&\sum_{j=1}^{\infty}\sum_{i=1}^{\infty}\langle \langle \Phi{\bf Q}^{\frac12}f_{i},e_j\rangle_V,\langle \Phi{\bf Q}^{\frac12}f_{i},e_j\rangle_V\rangle_{\mathcal{H}}\\
	=&\sum_{j=1}^{\infty}\sum_{i=1}^{\infty}\int_{0}^{T}\int_{0}^{T}\langle \Phi_u{\bf Q}^{\frac12}f_{i},e_j\rangle_V\langle \Phi_v{\bf Q}^{\frac12}f_{i},e_j\rangle_V\phi(u,v)dudv\\
	=&\sum_{i=1}^{\infty}\int_{0}^{T}\int_{0}^{T}\langle \Phi_u{\bf Q}^{\frac12}f_{i},\Phi_v{\bf Q}^{\frac12}f_{i}\rangle_V\phi(u,v)dudv\\
	=&\int_{0}^{T}\int_{0}^{T}\langle \Phi_u,\Phi_v\rangle_{\mathcal{L}^0_2}\phi(u,v)dudv.
	\end{align*}
\end{proof}

The following theorem provides the optimal regularity of the mild solution of \eqref{eq1}.

\begin{theorem}\label{tm-regu}
Let Assumptions \ref{x0}-\ref{Q} be satisfied with $\gamma\in(1,3-2H]$. Then \eqref{eq1} admits a unique mild solution such that 
\begin{align*}
\|X_t\|_{L^2(\Omega;\dot{V}^{2H+\gamma-1})}\le C\Big(1+\|u_0\|_{\dot{V}^{2H+\gamma-1}}\Big).
\end{align*}
Moreover, for any $\mu\in[0,2H+\gamma-1]$, it holds that
\begin{align*}
\|X_t-X_s\|_{L^2(\Omega;\dot{V}^{\mu})}\le C\Big(1+\|u_0\|_{\dot{V}^{2H+\gamma-1}}\Big)|t-s|^{\min\big\{\frac{2H+\gamma-1-\mu}{2},H\big\}}.
\end{align*}
\end{theorem}
\begin{proof}
	Under Assumptions \ref{x0}-\ref{Q}, the existence and uniqueness of the mild solution in $L^2(\Omega;V)$ follows from Lemma \ref{lm-ito} and the Gronwall's inequality. In the sequel, we concentrate on the proof for the regularity of the solution. 
	
	Suppose $0\le s<t\le T$. By means of the formulation \eqref{mild} and Assumption \ref{x0}, we get
	\begin{align*}
	&\|X_t\|_{L^2(\Omega;\dot{V}^{2H+\gamma-1})}\\
	\le &
	\|u_0\|_{\dot{V}^{2H+\gamma-1}}+\left\|\int_{0}^{t}S_{t-s}F(X_s)ds\right\|_{L^2(\Omega;\dot{V}^{2H+\gamma-1})}
		+\left\|\int_{0}^{t}S_{t-s}dW_s^{{\bf Q}}\right\|_{L^2(\Omega;\dot{V}^{2H+\gamma-1})}.
	\end{align*}
	Applying Lemmas \ref{lm-phi}-\ref{lm-ito} with $\rho=H$, we derive from Assumption \ref{Q} that
	\begin{align*}
	&\left\|\int_{0}^{t}S_{t-s}dW_s^{{\bf Q}}\right\|^2_{L^2(\Omega;\dot{V}^{2H+\gamma-1})}\\
	=&\int_{0}^{t}\int_{0}^{t}\langle A^HS_{t-u}A^{\frac{\gamma-1}{2}},A^HS_{t-v}A^{\frac{\gamma-1}{2}}\rangle_{\mathcal{L}^0_2}\phi(u,v)dudv\\
	=&\sum_{i=1}^{\infty}\int_{0}^{t}\int_{0}^{t}\langle A^HS_{t-u}A^{\frac{\gamma-1}{2}}{\bf Q}^\frac12 f_i,A^HS_{t-v}A^{\frac{\gamma-1}{2}}{\bf Q}^\frac12 f_i\rangle_V\phi(u,v)dudv\\
	\le&C \left\| A^{\frac{\gamma-1}{2}} \right\|^2_{\mathcal{L}^0_2}.
	\end{align*}
	If $\gamma\in(1,3-2H)$,
	Lemma \ref{lm-A} and Assumption \ref{F} lead to
	\begin{align}
	\left\|\int_{0}^{t}S_{t-s}F(X_s)ds\right\|_{L^2(\Omega;\dot{V}^{2H+\gamma-1})}\label{eq-SF}
	\le &\int_{0}^{t}\left\|A^{\frac{2H+\gamma-1}{2}}S_{t-s}F(X_s)\right\|_{L^2(\Omega;V)}ds\\
	\le &C\int_{0}^{t}|t-s|^{-\frac{2H+\gamma-1}{2}}\left\|F(X_s)\right\|_{L^2(\Omega;V)}ds\nonumber\\   
	\le &C\left(\int_{0}^{t}|t-s|^{-\frac{2H+\gamma-1}{2}}ds\right)\Big(1+\|u_0\|_{\dot{V}^{2H+\gamma-1}}\Big)\nonumber\\
	\le&C\Big(1+\|u_0\|_{\dot{V}^{2H+\gamma-1}}\Big).\nonumber
	\end{align}
	Then
	\begin{align*}
	\|X_t\|_{L^2(\Omega;\dot{V}^{2H+\gamma-1})}\le C\Big(1+\|u_0\|_{\dot{V}^{2H+\gamma-1}}\Big).
	\end{align*}
	
	Considering the temporal regularity, we have
	\begin{align*}
	&\|X_t-X_s\|_{L^2(\Omega;\dot{V}^{\mu})}\\
	\le &
	\|(S_{t-s}-{\rm Id}_V)X_s\|_{L^2(\Omega;\dot{V}^{\mu})}
	+\left\|\int_{s}^{t}S_{t-\sigma}F(X_\sigma)d\sigma\right\|_{L^2(\Omega;\dot{V}^{\mu})}
	+\left\|\int_{s}^{t}S_{t-\sigma}dW_\sigma^{{\bf Q}}\right\|_{L^2(\Omega;\dot{V}^{\mu})}.
	\end{align*}
	Applying Lemma \ref{lm-phi} with $\rho=\frac{\mu+1-\gamma}{2}\in[0,H]$, i.e.,  $\mu\in[\gamma-1,2H+\gamma-1]$, we obtain
    	\begin{align*}
    &\left\|\int_{s}^{t}S_{t-\sigma}dW_\sigma^{{\bf Q}}\right\|^2_{L^2(\Omega;\dot{V}^{\mu})}\\
    =&\sum_{i=1}^{\infty}\int_{s}^{t}\int_{s}^{t}\langle A^{\frac{\mu+1-\gamma}{2}}S_{t-u}A^{\frac{\gamma-1}{2}}{\bf Q}^\frac12 f_i,A^{\frac{\mu+1-\gamma}{2}}S_{t-v}A^{\frac{\gamma-1}{2}}{\bf Q}^\frac12 f_i\rangle_V\phi(u,v)dudv\\
    \le&C |t-s|^{\min\big\{2H+\gamma-1-\mu,2H\big\}}\left\| A^{\frac{\gamma-1}{2}} \right\|^2_{\mathcal{L}^0_2}.
    \end{align*}
    Lemma \ref{lm-A} implies that for any $\mu\in[0,2)$, 
    	\begin{align*}
    \left\|\int_{s}^{t}S_{t-\sigma}F(X_\sigma)d\sigma\right\|_{L^2(\Omega;\dot{V}^{\mu})}
    \le & \int_{s}^{t}\left\|A^{\frac{\mu}{2}}S_{t-\sigma}F(X_\sigma)\right\|_{L^2(\Omega;V)}d\sigma\\
    \le & C|t-s|^{\frac{2-\mu}{2}}\Big(1+\|u_0\|_{\dot{V}^{2H+\gamma-1}}\Big).
    \end{align*}
	Combining with 
	\begin{align*}
	&\|(S_{t-s}-{\rm Id}_V)X_s\|_{L^2(\Omega;\dot{V}^{\mu})}\\
	=&	\left\|A^{\frac{\mu}{2}}A^{-\frac{2H+\gamma-1}{2}}(S_{t-s}-{\rm Id}_V)A^{\frac{2H+\gamma-1}{2}}X_s\right\|_{L^2(\Omega;V)}\\
	\le & C|t-s|^{\frac{2H+\gamma-1-\mu}{2}}\|X_s\|_{L^2(\Omega;\dot{V}^{2H+\gamma-1})},
	\end{align*}
	we have
	\begin{align*}
	\|X_t-X_s\|_{L^2(\Omega;\dot{V}^{\mu})}\le C\Big(1+\|u_0\|_{\dot{V}^{2H+\gamma-1}}\Big)|t-s|^{\min\big\{\frac{2H+\gamma-1-\mu}{2},H\big\}},
	\end{align*}
	for $\gamma\in(1,3-2H)$ and $\mu\in[0,2H+\gamma-1]$.
	
	When $\gamma=3-2H$, it suffices to revise the estimate in \eqref{eq-SF}. The previous arguments yield a priori estimates that 
	\begin{align*}
&	\|X_t\|_{L^2(\Omega;\dot{V}^{1})}\le C\Big(1+\|u_0\|_{\dot{V}^{2}}\Big),\\
&	\|X_t-X_s\|_{L^2(\Omega;\dot{V}^{1})}\le C\Big(1+\|u_0\|_{\dot{V}^{2}}\Big)|t-s|^{\frac14}.
	\end{align*}
	Taking advantage of Lemma \ref{lm-A} and a priori estimates above, we deduce
	\begin{align*}
	&\left\|\int_{0}^{t}S_{t-s}F(X_s)ds\right\|_{L^2(\Omega;\dot{V}^{2})}\\
	\le &\left\|A\int_{0}^{t}S_{t-s}F(X_t)ds\right\|_{L^2(\Omega;V)}+\left\|A\int_{0}^{t}S_{t-s}(F(X_t)-F(X_s))ds\right\|_{L^2(\Omega;V)}\\
			\le & C \|F(X_t)\|_{L^2(\Omega;V)}+C \int_{0}^{t}(t-s)^{-1}\left\|F(X_t)-F(X_s)\right\|_{L^2(\Omega;V)} ds \\
					\le & C \Big(1+\|u_0\|_{\dot{V}^{2}}\Big) \left(1+ \int_{0}^{t}(t-s)^{-\frac34}  ds \right) \\ \le &C\Big(1+\|u_0\|_{\dot{V}^{2}}\Big),
	\end{align*}
	from which we complete the proof.
\end{proof}

\section{Spatial semi-discretizaiton}\label{sec4}
In this section, we study the spatial semi-discretization for \eqref{eq1}. First, based on the $K$-dimensional subspace $U_K:={\rm span}\{f_i:i=1,\cdots,K\}\subset U$, we truncate the infinite dimensional fBm to obtain the SHE driven by the $K$-dimensional fBm. We give the estimates for Malliavin derivatives of its mild solution and the strong convergence order associated with the truncation.
Next, we apply the spectral Galerkin method to spatially discretize the SHE driven by the $K$-dimensional fBm and show the optimal strong convergence rate of the spectral Galerkin method, which coincides with the optimal spatial regularity of the solution of \eqref{eq1}. Furthermore, in preparation for proving that the exponential integrator is super-convergent in time in the next section, we derive the temporal regularity of the mild solution in $L^4(\Omega;V)$ as a priori estimate.

Let
\begin{align*}
W^{{\bf Q},K}_t:=\sum_{i=1}^{K}{\bf Q}^{\frac12} f_i \beta^i_t
\end{align*}
be the truncation of the infinite dimensional fBm. 
Consider the mild solution $X^{K}_t$ of the SHE driven by the $K$-dimensional fBm
\begin{align}\label{eq2}
\left\{
\begin{aligned}
dX^{K}_t&=-A X^{K}_t dt+ F(X^{K}_t)dt+dW^{{\bf Q},K}_t,\quad t\in(0,T],\\
X^{K}_0&=u_0.
\end{aligned}
\right.
\end{align}
For any $x\in(0,1)$, $t\in(0,T]$ and $p\ge 1$, one knows from \cite[Proposition 3.3 and Lemma 3.4]{DT13} and \cite[Proposition 7]{NS09SPA} that the random variable $X^{K}_t(x)\in\mathbb{D}^{2,p}$. The estimates for the Malliavin derivatives are given as follows. 

\begin{lemma}\label{lm-esti-M}
	Let Assumptions \ref{x0}-\ref{F} be satisfied with $\gamma>1$.
	Denote $\varPsi^{i_1}_{t,r_1}(x):=\big(D_{r_1}X^{K}_t(x)\big)^{i_1}$ and 
	$\varPsi^{i_1,i_2}_{t,r_1,r_2}(x):=\big(D_{r_2}\big(D_{r_1}X^{K}_t(x)\big)^{i_1}\big)^{i_2}$, for $0\le r_1,r_2\le t$ and $1\le i_1,i_2\le K$. Then for any $0\le r<2$, there exists some constant $C=C(T,F,r)$ such that
	\begin{align*}
		\big\|\varPsi^{i_1}_{t,r_1}\big\|_{\dot{V}^{r}}&\le C\big\|{\bf Q}^{\frac12} f_{i_1}\big\|_{\dot{V}^{r}},\\
		\big\|\varPsi^{i_1,i_2}_{t,r_1,r_2}\big\|_V&\le  C\big\|{\bf Q}^{\frac12} f_{i_1}\big\|_V\big\|{\bf Q}^{\frac12} f_{i_2}\big\|_V.
	\end{align*}
\end{lemma}
\begin{proof}
	Based on \cite[Proposition 3.3 and Lemma 3.4]{DT13} and \cite[Proposition 7]{NS09SPA}, we have that $\varPsi^{i_1}_{t,r_1},\varPsi^{i_1,i_2}_{t,r_1,r_2}$ satisfy the linear equations
	\begin{align*}
		\varPsi^{i_1}_{t,r_1}&=S_{t-r_1}{\bf Q}^{\frac12} f_{i_1}+\int_{r_1}^{t}S_{t-s}\big(F'(X^{K}_s)\varPsi^{i_1}_{s,r_1} \big)ds,\\
		\varPsi^{i_1,i_2}_{t,r_1,r_2}&=\int_{r_1\vee r_2}^{t}S_{t-s}\big(F''(X^{K}_s)\varPsi^{i_1}_{s,r_1} \varPsi^{i_2}_{s,r_2} \big) ds+\int_{r_1\vee r_2}^{t}S_{t-s}\big(F'(X^{K}_s)\varPsi^{i_1,i_2}_{s,r_1,r_2} \big)ds,
	\end{align*}
	respectively. 
	The uniform boundedness of $\{S_{t}\}_{t\ge 0}$, Lemma \ref{lm-A} and Assumption \ref{F} lead to
	\begin{align*}
		\big\|\varPsi^{i_1}_{t,r_1}\big\|_{\dot{V}^{r}}\le& \big\|{\bf Q}^{\frac12} f_{i_1}\big\|_{\dot{V}^{r}}+C\int_{r_1}^{t}\big\|A^{\frac{r}{2}}S_{t-s}\big(F'(X^{K}_s)\varPsi^{i_1}_{s,r_1}\big) \big\|_Vds\\
		\le& \big\|{\bf Q}^{\frac12} f_{i_1}\big\|_{\dot{V}^{r}}+C\int_{r_1}^{t}|t-s|^{-\frac{r}{2}}\big\|F'(X^{K}_s)\varPsi^{i_1}_{s,r_1} \big\|_{V}ds\\
	\le& \big\|{\bf Q}^{\frac12} f_{i_1}\big\|_{\dot{V}^{r}}+C\int_{r_1}^{t}|t-s|^{-\frac{r}{2}}\big\|\varPsi^{i_1}_{s,r_1} \big\|_{V}ds.
	\end{align*}
	Then the Gronwall's inequality yields
		\begin{align*}
	\big\|\varPsi^{i_1}_{t,r_1}\big\|_{\dot{V}^{r}}&\le C\big\|{\bf Q}^{\frac12} f_{i_1}\big\|_{\dot{V}^{r}}.
	\end{align*}

For the second derivative, utilizing Lemma \ref{lm-A} and the Sobolev embedding $\dot{V}^{\alpha}\hookrightarrow L^\infty([0,1])$, $\frac12<\alpha<1$, we have
\begin{align*}
\big\|S_{t-s}\big(F''(X^{K}_s)\varPsi^{i_1}_{s,r_1} \varPsi^{i_2}_{s,r_2}\big)\big\|_V=&
\big\|A^{\frac{\alpha}{2}}S_{t-s}A^{-\frac{\alpha}{2}}\big(F''(X^{K}_s)\varPsi^{i_1}_{s,r_1} \varPsi^{i_2}_{s,r_2}\big)\big\|_V\\
\le&
|t-s|^{-\frac{\alpha}{2}}\big\|F''(X^{K}_s)\varPsi^{i_1}_{s,r_1} \varPsi^{i_2}_{s,r_2}\big\|_{\dot{V}^{-\alpha}}\\
\le& |t-s|^{-\frac{\alpha}{2}}\int_{0}^{1} \big|f''(X^{K}_s(x))\varPsi^{i_1}_{s,r_1} (x)\varPsi^{i_2}_{s,r_2}(x)\big| dx\\
\le& |t-s|^{-\frac{\alpha}{2}}\big\|\varPsi^{i_1}_{s,r_1}\big\|_V\big\|\varPsi^{i_2}_{s,r_2} \big\|_V.
\end{align*}
Therefore, we obtain
	\begin{align*}
		&\big\|\varPsi^{i_1,i_2}_{t,r_1,r_2}\big\|_V\\
		\le &C\int_{r_1\vee r_2}^{t}|t-s|^{-\frac{\alpha}{2}}\big\|\varPsi^{i_1}_{s,r_1}\big\|_V\big\|\varPsi^{i_2}_{s,r_2} \big\|_Vds+C\int_{r_1\vee r_2}^{t}\big\|\varPsi^{i_1,i_2}_{s,r_1,r_2}\big\|_Vds\\
	\le&  C\int_{r_1\vee r_2}^{t}|t-s|^{-\frac{\alpha}{2}}\big\|{\bf Q}^{\frac12} f_{i_1}\big\|_V\big\|{\bf Q}^{\frac12} f_{i_2}\big\|_Vds+C\int_{r_1\vee r_2}^{t}\big\|\varPsi^{i_1,i_2}_{s,r_1,r_2}\big\|_Vds\\
		\le & C\big\|{\bf Q}^{\frac12} f_{i_1}\big\|_V\big\|{\bf Q}^{\frac12} f_{i_2}\big\|_V+C\int_{r_1\vee r_2}^{t}\big\|\varPsi^{i_1,i_2}_{s,r_1,r_2}\big\|_Vds\\
	\le &C\big\|{\bf Q}^{\frac12} f_{i_1}\big\|_V\big\|{\bf Q}^{\frac12} f_{i_2}\big\|_V ,
	\end{align*}
	where the Gronwall's inequality is used in the last inequality.
\end{proof}

Denote $\mathcal{P}_K$ as the project operator from $U$ onto $U_K$. We give the strong error analysis for the SHE driven by the $K$-dimensional fBm.

\begin{theorem}\label{tm-3}
	Let $X$ and $X^K$ be mild solutions of \eqref{eq1} and \eqref{eq2}, respectively. Under Assumptions \ref{x0}-\ref{Q} with $\gamma>1$, it holds that
	\begin{align*}
	\sup_{t\in[0,T]}\left\| X_{t}-X^K_{t}  \right\|_{L^2(\Omega;V)}\le C\big\|{\bf Q}^\frac12 ({\rm Id}_U-\mathcal{P}_K)\big\|_{\mathcal{L}_2(U,\dot{V}^{-2H})}.
	\end{align*}
\end{theorem}
\begin{proof}
	The error has the following decomposition
	\begin{align*}
		&\left\| X_{t}-X^K_{t}  \right\|_{L^2(\Omega;V)}\\
		\le &\left\| \int_{0}^{t} S_{t-s}\big(F(X_s)-F(X^K_s)\big)ds \right\|_{L^2(\Omega;V)}+\left\| \int_{0}^{t} S_{t-s}dW^{{\bf Q}}_s-\int_{0}^{t} S_{t-s}dW^{{\bf Q},K}_s\right\|_{L^2(\Omega;V)}.
	\end{align*}
	Using Lemmas \ref{lm-phi}-\ref{lm-ito} with $\rho=H$, we get 
	\begin{align*}
	&\mathbb{E}\left[\left\|  \int_{0}^{t} S_{t-s}dW^{{\bf Q}}_s-\int_{0}^{t} S_{t-s}dW^{{\bf Q},K}_s  \right\|^2_V\right]\\
	=&\sum_{i=K+1}^{\infty}\int_{0}^{t}\int_{0}^{t}\langle S_{t-u}{\bf Q}^\frac12 f_i,S_{t-v}{\bf Q}^\frac12 f_i\rangle_V\phi(u,v)dudv\\
	=&\sum_{i=K+1}^{\infty}\int_{0}^{t}\int_{0}^{t}\langle A^HS_{t-u}A^{-H}{\bf Q}^\frac12 f_i,A^{H}S_{t-v}A^{-H}{\bf Q}^\frac12 f_i\rangle_V\phi(u,v)dudv\\
	\le& C \sum_{i=K+1}^{\infty}\big\|A^{-H}{\bf Q}^\frac12 f_i\big\|_V^2\\
	=&C \big\|{\bf Q}^\frac12 ({\rm Id}_U-\mathcal{P}_K)\big\|_{\mathcal{L}_2(U,\dot{V}^{-2H})}^2.
	\end{align*}
	Hence, 
	\begin{align*}
			\left\| X_{t}-X^K_{t}  \right\|_{L^2(\Omega;V)}
			\le C\int_{0}^{t} \left\|X_s-X^K_s \right\|_{L^2(\Omega;V)}ds
	+C \big\|{\bf Q}^\frac12 ({\rm Id}_U-\mathcal{P}_K)\big\|_{\mathcal{L}_2(U,\dot{V}^{-2H})},
	\end{align*}
	from which we conclude the result by the Gronwall's inequality.
\end{proof}

In the next step, we spatially discretize \eqref{eq2} by the spectral Galerkin method. More precisely, denoting by  $P_M$ the project operator from $V$ onto $V_M:={\rm span}\{e_i:i=1,\cdots,M\}\subset V$ and $A_M:=AP_M$, we obtain
\begin{align}\label{eq3}
\left\{
\begin{aligned}
dX^{K,M}_t&=-A_M X^{K,M}_t dt+ P_MF(X^{K,M}_t)dt+P_MdW^{{\bf Q},K}_t,\quad t\in(0,T],\\
X^{K,M}_0&=P_Mu_0.
\end{aligned}
\right.
\end{align}
The optimal strong convergence rate of the spectral Galerkin method is proved as follows.

\begin{theorem}\label{tm-4}
	Suppose that $X^K$ and $X^{K,M}$ are mild solutions of \eqref{eq2} and \eqref{eq3}, respectively. Under Assumptions \ref{x0}-\ref{Q} with $\gamma\in(1,3-2H]$, it holds that
	\begin{align*}
	\sup_{t\in[0,T]}\left\| X^K_{t}-X^{K,M}_{t}  \right\|_{L^2(\Omega;V)}\le C\Big(1+\|u_0\|_{\dot{V}^{2H+\gamma-1}}\Big) M^{-(2H+\gamma-1)}.
	\end{align*}
\end{theorem}
\begin{proof}
	Define by $\{S^M_{t}\}_{t\ge 0}$ the analytic semigroup generated by $-A_M$. 
	We decompose the error by
	\begin{align*}
	&\left\| X^K_{t}-X^{K,M}_{t}  \right\|_{L^2(\Omega;V)}\\
	\le & \left\| ({\rm Id_V}-P_M)u_0\right\|_{L^2(\Omega;V)}
	+\left\| \int_{0}^{t} \big(S_{t-s}-S^M_{t-s}\big)F(X^K_s)ds \right\|_{L^2(\Omega;V)}\\
		&+\left\| \int_{0}^{t} \big(S^M_{t-s}F(X^K_s)-S^M_{t-s}F(X^{K,M}_s)\big)ds \right\|_{L^2(\Omega;V)}
	+\left\| \int_{0}^{t} \big(S_{t-s}- S^M_{t-s}\big)dW^{{\bf Q},K}_s\right\|_{L^2(\Omega;V)}\\
	\le & C\|u_0\|_{\dot{V}^{2H+\gamma-1}}M^{-(2H+\gamma-1)}
		+ \int_{0}^{t}\left\| \big(S_{t-s}-S^M_{t-s}\big)F(X^K_s) \right\|_{L^2(\Omega;V)}ds\\
	&+\int_{0}^{t} \Big\| X^K_s-X^{K,M}_s\Big\|_{L^2(\Omega;V)}ds 
	+\left\| \int_{0}^{t} \big(S_{t-s}- S^M_{t-s}\big)dW^{{\bf Q},K}_s\right\|_{L^2(\Omega;V)}.
	\end{align*}
By Lemma \ref{lm-A} and Theorem \ref{tm-regu}, we get
\begin{align*}
		&\int_{0}^{t} \left\| \big(S_{t-s}-S^M_{t-s}\big)F(X^K_s)\right\|_{L^2(\Omega;V)}ds \\
				\le & \int_{0}^{t} \left\|\big(S_{t-s}-S^M_{t-s}\big)F(X^K_t)\right\|_{L^2(\Omega;V)}ds 
				+ \int_{0}^{t} \left\|\big(S_{t-s}-S^M_{t-s}\big)\big(F(X^K_t)-F(X^K_s)\big)\right\|_{L^2(\Omega;V)}ds \\
				\le & 				\int_{0}^{t} \left\|({\rm Id}_V-P_M)F(X^K_t)\right\|_{L^2(\Omega;V)}ds \\
				&+ \int_{0}^{t} \left\|A^{\frac{2H+\gamma-1}{2}}S_{t-s}({\rm Id}_V-P_M)A^{\frac{-(2H+\gamma-1)}{2}}\big(F(X^K_t)-F(X^K_s)\big)\right\|_{L^2(\Omega;V)}ds \\
				\le & 			CM^{-(2H+\gamma-1)}	 \left\|F(X^K_t)\right\|_{L^2(\Omega;\dot{V}^{2H+\gamma-1})} \\
				&+ CM^{-(2H+\gamma-1)}	\int_{0}^{t} (t-s)^{\frac{-(2H+\gamma-1)}{2}} \left\|F(X^K_t)-F(X^K_s)\right\|_{L^2(\Omega;V)}ds \\
				\le& CM^{-(2H+\gamma-1)}.
\end{align*}
According to Lemma \ref{lm-ito} and \cite[Lemma 4.1]{wang17BIT}, we have 
\begin{align*}
&\left\| \int_{0}^{t} \big(S_{t-s}- S^M_{t-s}\big)dW^{{\bf Q},K}_s\right\|^2_{L^2(\Omega;V)}\\
=&\sum_{i=1}^{K}\int_{0}^{t}\int_{0}^{t}\langle \big(S_{t-u}- S^M_{t-u}\big){\bf Q}^\frac12 f_i,\big(S_{t-v}- S^M_{t-v}\big){\bf Q}^\frac12 f_i\rangle _V\phi(u,v)dudv\\
\le & C\sum_{i=1}^{\infty}M^{-2(2H+\gamma-1)}\big\|{\bf Q}^\frac12 f_i\big\|^2_{\dot{V}^{\gamma-1}}\\
\le & CM^{-2(2H+\gamma-1)}\left\|A^{\frac{\gamma-1}{2}}\right\|^2_{\mathcal{L}^0_2}.
\end{align*}
	Then the strong convergence order of the spectral Galerkin method in spatial direction is a consequence of the Gronwall's inequality.
\end{proof}

\begin{remark}
Suppose $U=V$, $K=M$ and $W^{\bf Q}_t=\sum_{i=1}^{\infty}\eta_ie_i\beta_t^i$.  One can utilize $P_M$ to perform the spectral Galerkin method in terms of the noise $W^{\bf Q}$ and the solution $X$ at the same time,
with the spatial convergence rate $\mathcal{O}\big(M^{-(2H+\gamma-1)}\big)$.
\end{remark}

The following lemma gives the estimates for the fourth-order moment of the stochastic integral with respect to the $K$-dimensional fBm. Based on these estimates, we obtain the temporal regularity of $X^{K,M}$ in $L^4(\Omega;V)$.

\begin{lemma}\label{lm-regu4}
	Let Assumption \ref{Q} be satisfied with $\gamma>3-2H$. Then 
	\begin{align*}
	&\left\|\int_{0}^{t}S^M_{t-s}dW_s^{{\bf Q},K}\right\|_{L^4\left(\Omega;\dot{V}^{\frac{2H+\gamma-1}{2}}\right)}\le C,\\
	&\left\|\int_{s}^{t}S^M_{t-\sigma}dW_\sigma^{{\bf Q},K}\right\|_{L^4(\Omega;V)}\le C|t-s|^{\frac12}.
	\end{align*}
\end{lemma}
\begin{proof}
	By the definition of $W^{{\bf Q},K}$, we have
	\begin{align*}
	&\left\|\int_{0}^{t}S^M_{t-s}dW_s^{{\bf Q},K}\right\|_{L^4\left(\Omega;\dot{V}^{\frac{2H+\gamma-1}{2}}\right)}\\
	%=&\left\|\sum_{j=1}^{K}\int_{0}^{t}S^M_{t-s}{\bf Q}^\frac12f_jd\beta^j_s\right\|_{L^4\left(\Omega;\dot{V}^{\frac{2H+\gamma-1}{2}}\right)}\\
	=&\left\|\sum_{j=1}^{K}\int_{0}^{t}A^{\frac{2H+\gamma-1}{4}}S^M_{t-s}{\bf Q}^\frac12f_jd\beta^j_s\right\|_{L^4\left(\Omega;V\right)}\\
	%=&\left\|\sum_{j=1}^{K}\int_{0}^{t}\sum_{i=1}^{M}\langle A^{\frac{2H+\gamma-1}{4}}S^M_{t-s}{\bf Q}^\frac12f_j,e_i\rangle _Ve_id\beta^j_s\right\|_{L^4\left(\Omega;V\right)}\\
	=&\left\|\sum_{i=1}^{M}\sum_{j=1}^{K}\int_{0}^{t}\langle A^{\frac{2H+\gamma-1}{4}}S^M_{t-s}{\bf Q}^\frac12f_j,e_i\rangle _Vd\beta^j_se_i\right\|_{L^4\left(\Omega;V\right)}\\
	\le &\sum_{i=1}^{M}\left\|\sum_{j=1}^{K}\int_{0}^{t}\langle A^{\frac{2H+\gamma-1}{4}}S^M_{t-s}{\bf Q}^\frac12f_j,e_i\rangle _Vd\beta^j_se_i\right\|_{L^4\left(\Omega;V\right)}\\
	\le&\sum_{i=1}^{M}\left\|\sum_{j=1}^{K}\int_{0}^{t}\langle A^{\frac{2H+\gamma-1}{4}}S^M_{t-s}{\bf Q}^\frac12f_j,e_i\rangle _Vd\beta^j_s\right\|_{L^4\left(\Omega;\mathbb{R}\right)}.
	\end{align*}
	Denoting
	$$B^K_s=(\beta^1_s,\cdots,\beta_s^K)$$
	and 
		$$\varphi_s=\Big(\langle A^{\frac{2H+\gamma-1}{4}}S^M_{t-s}{\bf Q}^\frac12f_1,e_i\rangle _V,\cdots,\langle A^{\frac{2H+\gamma-1}{4}}S^M_{t-s}{\bf Q}^\frac12f_K,e_i\rangle _V\Big)\in\mathbb{R}^K,$$
	we derive from Lemma \ref{Lp} that
	\begin{align*}
	&\left\|\int_{0}^{t}S^M_{t-s}dW_s^{{\bf Q},K}\right\|_{L^4\left(\Omega;\dot{V}^{\frac{2H+\gamma-1}{2}}\right)}\\
	\le&\sum_{i=1}^{M}\left\|\int_{0}^{t}\varphi_s\delta B^K_s\right\|_{L^4\left(\Omega;\mathbb{R}\right)}\\
	\le &C\sum_{i=1}^{M} \Bigg[\sum_{j=1}^{K}\int_{0}^{t}\int_{0}^{t}\left|\langle A^{\frac{2H+\gamma-1}{4}}S^M_{t-u}{\bf Q}^\frac12f_j,e_i\rangle _V\right|
	\left|\langle A^{\frac{2H+\gamma-1}{4}}S^M_{t-v}{\bf Q}^\frac12f_j,e_i\rangle _V\right| \phi(u,v)dudv\Bigg]^\frac12.
	\end{align*}
	Under Assumption \ref{Q}, we obtain
	\begin{align*}
	&\sum_{j=1}^{K}\int_{0}^{t}\int_{0}^{t}\left|\langle A^{\frac{2H+\gamma-1}{4}}S^M_{t-u}{\bf Q}^\frac12f_j,e_i\rangle _V\right|
	\left|\langle A^{\frac{2H+\gamma-1}{4}}S^M_{t-v}{\bf Q}^\frac12f_j,e_i\rangle _V\right| \phi(u,v)dudv\\
	\le &\sum_{j=1}^{\infty}\left\|A^{\frac{\gamma-1}{2}}{\bf Q}^\frac12f_j\right\|^2_V\int_{0}^{t}\int_{0}^{t}\lambda_i^{\frac{2H-\gamma+1}{2}}e^{-\lambda_i(t-u)}e^{-\lambda_i(t-v)} \phi(u,v)dudv\\
	= &\left\| A^{\frac{\gamma-1}{2}} \right\|^2_{\mathcal{L}^0_2}\int_{0}^{t}\int_{0}^{t}\lambda_i^{\frac{2H-\gamma+1}{2}}e^{-\lambda_i(t-u)}e^{-\lambda_i(t-v)} \phi(u,v)dudv.
	\end{align*}
	Hence,
	\begin{align*}
	&\left\|\int_{0}^{t}S^M_{t-s}dW_s^{{\bf Q},K}\right\|_{L^4\left(\Omega;\dot{V}^{\frac{2H+\gamma-1}{2}}\right)}\\
	\le &C\left\| A^{\frac{\gamma-1}{2}} \right\|_{\mathcal{L}^0_2}\sum_{i=1}^{M} \Bigg[\int_{0}^{t}\int_{0}^{t}\lambda_i^{\frac{2H-\gamma+1}{2}}e^{-\lambda_i(t-u)}e^{-\lambda_i(t-v)} \phi(u,v)dudv\Bigg]^\frac12.
	\end{align*}
	Similarly, 
	\begin{align*}
	&\left\|\int_{s}^{t}S^M_{t-\sigma}dW_\sigma^{{\bf Q},K}\right\|_{L^4(\Omega;V)}\\
	\le &C\left\| A^{\frac{\gamma-1}{2}} \right\|_{\mathcal{L}^0_2}\sum_{i=1}^{M} \Bigg[\int_{s}^{t}\int_{s}^{t}\lambda_i^{-(\gamma-1)}e^{-\lambda_i(t-u)}e^{-\lambda_i(t-v)} \phi(u,v)dudv\Bigg]^\frac12.
	\end{align*}
	
	It suffices to show
	\begin{align*}
	\sum_{i=1}^{\infty} \Bigg[\int_{0}^{t}\int_{0}^{t}\lambda_i^{\frac{2H-\gamma+1}{2}}e^{-\lambda_i(t-u)}e^{-\lambda_i(t-v)} \phi(u,v)dudv\Bigg]^\frac12
	\le  C
	\end{align*}
	and 
	\begin{align*}
	\sum_{i=1}^{\infty} \Bigg[\int_{s}^{t}\int_{s}^{t}\lambda_i^{-(\gamma-1)}e^{-\lambda_i(t-u)}e^{-\lambda_i(t-v)} \phi(u,v)dudv\Bigg]^\frac12
	\le C|t-s|^{\frac12}.
	\end{align*}
	Indeed, since $H\in(\frac12,1)$, there exists some $p>1$ such that $(2H-2)p>-1$. Applying the H\"older inequality with $\frac{1}{p}+\frac{1}{q}=1$, we obtain
	\begin{align*}
	&\int_{0}^{t}\int_{0}^{t}e^{-\lambda_i(t-u)}e^{-\lambda_i(t-v)} \phi(u,v)dudv\\
	=&\int_{0}^{t}e^{-\lambda_i(t-v)}\left[\int_{0}^{t}e^{-\lambda_i(t-u)} \phi(u,v)du\right]dv\\
	=&\int_{0}^{t}e^{-\lambda_i(t-v)}\left[\int_{0}^{t}e^{-\lambda_i(t-u)q} du\right]^{\frac1q}\left[\int_{0}^{t}\phi(u,v)^pdu\right]^{\frac1p} dv\\
	\le &C\int_{0}^{t}e^{-\lambda_i(t-v)}\left[\int_{0}^{t}e^{-\lambda_i(t-u)q} du\right]^{\frac1q} dv
	\le C\frac{1}{\lambda_i}\left(\frac{1}{q\lambda_i}\right)^{\frac1q}.
	\end{align*}
	Let $\alpha=\gamma-(3-2H)$. Taking $0<\epsilon<\min\big\{2H-1,\frac{\alpha}{2}\big\}$ and $p=\frac{1}{2-2H+\epsilon}$, we have $p>1$, $(2H-2)p>-1$ and $\frac{1}{q}=2H-1-\epsilon$, which leads to 
	\begin{align*}
	&\sum_{i=1}^{\infty} \Bigg[\int_{0}^{t}\int_{0}^{t}\lambda_i^{\frac{2H-\gamma+1}{2}}e^{-\lambda_i(t-u)}e^{-\lambda_i(t-v)} \phi(u,v)dudv\Bigg]^\frac12\\
	\le &C\sum_{i=1}^{\infty} i^{\frac{2H-\gamma+1}{2}}\frac{1}{i}\left(\frac{1}{i}\right)^{\frac1q}=C\sum_{i=1}^{\infty} i^{-1+\epsilon-\frac{\alpha}{2}}\le C.
	\end{align*}
	Similarly, based on 
	\begin{align*}
	&\int_{s}^{t}\int_{s}^{t}e^{-\lambda_i(t-u)}e^{-\lambda_i(t-v)} \phi(u,v)dudv\\
	=&\int_{s}^{t}e^{-\lambda_i(t-v)}\left[\int_{s}^{t}e^{-\lambda_i(t-u)} \phi(u,v)du\right]dv\\
	=&\int_{s}^{t}e^{-\lambda_i(t-v)}\left[\int_{s}^{t}e^{-\lambda_i(t-u)q} du\right]^{\frac1q}\left[\int_{s}^{t}\phi(u,v)^pdu\right]^{\frac1p} dv\\
	\le &C\int_{s}^{t}e^{-\lambda_i(t-v)}\left[\int_{s}^{t}e^{-\lambda_i(t-u)q} du\right]^{\frac1q} dv
	\le C|t-s|\left(\frac{1}{q\lambda_i}\right)^{\frac1q},
	\end{align*}
	we obtain
	\begin{align*}
	&\sum_{i=1}^{\infty} \Bigg[\int_{s}^{t}\int_{s}^{t}\lambda_i^{-(\gamma-1)}e^{-\lambda_i(t-u)}e^{-\lambda_i(t-v)} \phi(u,v)dudv\Bigg]^\frac12\\
	\le & C\sum_{i=1}^{\infty} |t-s|^{\frac12}i^{-(\gamma-1)}\left(\frac{1}{i}\right)^{\frac1q}\le  C\sum_{i=1}^{\infty} |t-s|^{\frac12}i^{-1+\epsilon-\alpha}\le C |t-s|^{\frac12}.
	\end{align*}
	This finishes the proof.
\end{proof}

\begin{theorem}\label{tm-regu4}
	Under Assumptions \ref{x0}-\ref{Q} with $\gamma\in(3-2H,5-2H]$, the mild solution of \eqref{eq3} satisfies
	\begin{align*}
	&\left\|X^{K,M}_t\right\|_{L^4\left(\Omega;\dot{V}^{\frac{2H+\gamma-1}{2}}\right)}\le C\Big(1+\|u_0\|_{\dot{V}^{2H+\gamma-1}}\Big),\\
	&\left\|X^{K,M}_t-X^{K,M}_s\right\|_{L^4(\Omega;V)}\le C\Big(1+\|u_0\|_{\dot{V}^{2H+\gamma-1}}\Big)|t-s|^{\frac12}.
	\end{align*}
\end{theorem}
\begin{proof}
	According to Lemma \ref{lm-regu4}, we have
	\begin{align*}
	\left\|\int_{0}^{t}S^M_{t-s}dW_s^{{\bf Q},K}\right\|_{L^4(\Omega;V)}\le C.
	\end{align*}
    Then the Gronwall's inequality yields
	\begin{align*}
	\left\|X^{K,M}_t\right\|_{L^4(\Omega;V)}\le C\Big(1+\|u_0\|_{\dot{V}^{2H+\gamma-1}}\Big).
	\end{align*}
	It follows from Lemma \ref{lm-A} and Lemma \ref{lm-regu4} that for $\gamma\in(3-2H,5-2H)$, 
	\begin{align*}
	&\left\|X^{K,M}_t\right\|_{L^4\left(\Omega;\dot{V}^{\frac{2H+\gamma-1}{2}}\right)}\\
	\le &
	\|u_0\|_{\dot{V}^{2H+\gamma-1}}
	+\left\|\int_{0}^{t}S^M_{t-s}F(X^{K,M}_s)ds\right\|_{L^4\left(\Omega;\dot{V}^{\frac{2H+\gamma-1}{2}}\right)}
	+\left\|\int_{0}^{t}S^M_{t-s}dW_s^{{\bf Q},K}\right\|_{L^4\left(\Omega;\dot{V}^{\frac{2H+\gamma-1}{2}}\right)}\\
	\le &
	\|u_0\|_{\dot{V}^{2H+\gamma-1}}+\int_{0}^{t}\left\|A^{\frac{2H+\gamma-1}{4}}S^M_{t-s}F(X^{K,M}_s)\right\|_{L^4(\Omega;V)}ds
	+\left\|\int_{0}^{t}S^M_{t-s}dW_s^{{\bf Q},K}\right\|_{L^4\left(\Omega;\dot{V}^{\frac{2H+\gamma-1}{2}}\right)}\\
	\le &
	\|u_0\|_{\dot{V}^{2H+\gamma-1}}
	+\int_{0}^{t}|t-s|^{-\frac{2H+\gamma-1}{4}}\left\|F(X^{K,M}_s)\right\|_{L^4(\Omega;V)}ds
	+\left\|\int_{0}^{t}S^M_{t-s}dW_s^{{\bf Q},K}\right\|_{L^4\left(\Omega;\dot{V}^{\frac{2H+\gamma-1}{2}}\right)}\\
	\le& C\Big(1+\|u_0\|_{\dot{V}^{2H+\gamma-1}}\Big).
	\end{align*}
	Furthermore, we get from Lemma \ref{lm-regu4} that 
	\begin{align*}
	&\left\|X^{K,M}_t-X^{K,M}_s\right\|_{L^4(\Omega;V)}\\
	\le &
	\Big\|\big(S^M_{t-s}-{\rm Id}_V\big)X^{K,M}_s\Big\|_{L^4(\Omega;V)}
	+\left\|\int_{s}^{t}S^M_{t-\sigma}F(X^{K,M}_\sigma)d\sigma\right\|_{L^4(\Omega;V)}
	+\left\|\int_{s}^{t}S^M_{t-\sigma}dW_\sigma^{{\bf Q},K}\right\|_{L^4(\Omega;V)}\\
	\le & C|t-s|^{\frac{2H+\gamma-1}{4}\wedge \frac12}\Big\|X^{K,M}_s\Big\|_{L^4\left(\Omega;\dot{V}^{\frac{2H+\gamma-1}{2}\wedge \frac12}\right)}
	+C|t-s|\Big(1+\left\|u_0\right\|_{L^4(\Omega;V)}\Big)
	+C|t-s|^\frac12\\
	\le&C\Big(1+\|u_0\|_{\dot{V}^{2H+\gamma-1}}\Big)|t-s|^\frac12.
	\end{align*}
	If $\gamma=5-2H$, using Lemma \ref{lm-A}, we get
	\begin{align*}
	&\left\|\int_{0}^{t}S^M_{t-s}F(X^{K,M}_s)ds\right\|_{L^4\left(\Omega;\dot{V}^{\frac{2H+\gamma-1}{2}}\right)}\\
	\le &\left\|\int_{0}^{t}S^M_{t-s}F(X^{K,M}_t)ds\right\|_{L^4\left(\Omega;\dot{V}^{\frac{2H+\gamma-1}{2}}\right)}\\
	&+\left\|\int_{0}^{t}S^M_{t-s}\big(F(X^{K,M}_t)-F(X^{K,M}_s)\big)ds\right\|_{L^4\left(\Omega;\dot{V}^{\frac{2H+\gamma-1}{2}}\right)}\\
	\le&C\left\|X^{K,M}_t\right\|_{L^4\left(\Omega;V\right)}+\int_{0}^{t}|t-s|^{-\frac{2H+\gamma-1}{4}}\left\|X^{K,M}_t-X^{K,M}_s\right\|_{L^4(\Omega;V)}ds\\
	\le &C\Big(1+\|u_0\|_{\dot{V}^{2H+\gamma-1}}\Big),
	\end{align*}
	from which we complete the proof.
\end{proof}

%\begin{remark}
%	If $\gamma=3-2H+\alpha$, $f_j=e_j$ and ${\bf Q}e_j=\eta_je_j=j^\eta e_j$, then $\left\|A^{\frac{\gamma-1}{2}}\right\|_{\mathcal{L}_2^0}<\infty$ if and only if
%	$\eta<4H-5-2\alpha$.
%\end{remark}

\section{Full discretization}\label{sec5}
In this section, we apply the exponential integrator to \eqref{eq3} to construct a fully discrete scheme, which is 
\begin{align}\label{eq4}
	\left\{
	\begin{aligned}
		X^{K,M,N}_{t_{n}}&:=S^M_{h}X^{K,M,N}_{t_{n-1}}+ S^M_{h}F(X^{K,M,N}_{t_{n-1}})h+S^M_{h}\Delta W^{{\bf Q},K}_{n},\\
		X^{K,M,N}_0&:=P_Mu_0,
\end{aligned}
\right.
\end{align}
where  $n=1,\cdots,N$, $h:=\frac{T}{N}$, $N\in \mathbb{N}_+$ and $\Delta W^{{\bf Q},K}_{n}:= W^{{\bf Q},K}_{t_n}- W^{{\bf Q},K}_{t_{n-1}}$.
		
Now we are in the position to prove that the exponential integrator is super-convergent in time, together with Lemma \ref{J1}, where the Malliavin calculus is an essential tool.
		
\begin{theorem}\label{lm-tem}\label{tm-6}
	Suppose that $X^{K,M}$ is the mild solution of \eqref{eq3} and that $X^{K,M,N}$ is defined by scheme \eqref{eq4}. Under Assumptions \ref{x0}-\ref{Q} with $\gamma>\max\{3-2H,\frac32\}$, it holds that
	\begin{align*}
	\sup_{n=0,\cdots,N}\left\| X^{K,M}_{t_n} -X^{K,M,N}_{t_{n}} \right\|_{L^2(\Omega;V)}\le C\Big(1+\|u_0\|_{\dot{V}^{2H+\gamma-1}}\Big)^2 N^{-1}.
	\end{align*}
\end{theorem}
\begin{proof}
	We rewrite scheme \eqref{eq4} as
\begin{align*}
X^{K,M,N}_{t_{n}}
=S^M_{t_n}P_Mu_0+\sum_{i=0}^{n-1}\int_{t_i}^{t_{i+1}}S^M_{t_n-t_i}F(X^{K,M,N}_{t_{i}}) ds
+\sum_{i=0}^{n-1}\int_{t_i}^{t_{i+1}}S^M_{t_n-t_i}d W^{{\bf Q},K}_s.
\end{align*}
Introducing the notation $\lfloor t \rfloor:=\max\{t_n:t_n\le t,n=0,\cdots,N\}$, we decompose the error into four parts 
\begin{align*}
&\left\| X^{K,M}_{t_n} -X^{K,M,N}_{t_{n}} \right\|_{L^2(\Omega;V)}\\
\le &	\left\| \int_{0}^{t_n}\Big( S^M_{t_n-s}F(X^{K,M}_s) - S^M_{t_n-\lfloor s \rfloor }F(X^{K,M,N}_{\lfloor s \rfloor}) \Big)ds \right\|_{L^2(\Omega;V)}\\
&+\left\| \int_{0}^{t_n} \big(S^M_{t_n-s}- S^M_{t_n-\lfloor s \rfloor} \big)dW^{{\bf Q},K}_s\right\|_{L^2(\Omega;V)}\\
\le &	\left\| \int_{0}^{t_n} S^M_{t_n-s}\Big(F(X^{K,M}_s)- F(X^{K,M}_{\lfloor s \rfloor}) \Big)ds \right\|_{L^2(\Omega;V)}\\
 &	+\left\| \int_{0}^{t_n}\big( S^M_{t_n-s}-S^M_{t_n-\lfloor s \rfloor}\big)F(X^{K,M}_{\lfloor s \rfloor}) ds \right\|_{L^2(\Omega;V)}\\
  &	+\left\| \int_{0}^{t_n} S^M_{t_n-\lfloor s \rfloor}\Big(F(X^{K,M}_{\lfloor s \rfloor}) -F(X^{K,M,N}_{\lfloor s \rfloor}) \Big)ds \right\|_{L^2(\Omega;V)}\\
&+\left\| \int_{0}^{t_n} \big(S^M_{t_n-s}- S^M_{t_n-\lfloor s \rfloor} \big)dW^{{\bf Q},K}_s\right\|_{L^2(\Omega;V)}\\
=&:J_1+J_2+J_3+J_4.
\end{align*}

The estimate for $J_1$ is postponed to Lemma \ref{J1} which gives
	\begin{align*}
J_1\le C\Big(1+\|u_0\|_{\dot{V}^{2H+\gamma-1}}\Big) h.
\end{align*}

For the second part $J_2$, note that
\begin{align*}
J_2\le &
\left\| \int_{0}^{t_n}\big( S^M_{t_n-s}-S^M_{t_n-\lfloor s \rfloor} \big)\Big(F(X^{K,M}_{t_{n-1}})-F(X^{K,M}_{\lfloor s \rfloor}\Big) ds \right\|_{L^2(\Omega;V)}\\
&+\left\| \int_{0}^{t_n}\big( S^M_{t_n-s}-S^M_{t_n-\lfloor s \rfloor} \big)F(X^{K,M}_{t_{n-1}}) ds \right\|_{L^2(\Omega;V)}\\
=&:J_{21}+J_{22}.
\end{align*}
Using the temporal regularity  shown in Theorem \ref{tm-regu}, we get
\begin{align*}
J_{21}&\le C\int_{0}^{t_n}|t_n-s|^{-1} |s-\lfloor s \rfloor|\left\|F(X^{K,M}_{t_{n-1}}) - F(X^{K,M}_{\lfloor s \rfloor}) \right\|_{L^2(\Omega;V)}ds\\
&\le C\Big(1+\|u_0\|_{\dot{V}^{2}}\Big) \int_{0}^{t_n}|t_n-s|^{-1} |s-\lfloor s \rfloor| |t_n-s|^Hds\\
&\le  C\Big(1+\|u_0\|_{\dot{V}^{2}}\Big) h.
\end{align*}
Together with 
\begin{align*}
  \int_{0}^{t_n}  e^{-\lambda_j(t_n-s)}-e^{-\lambda_j(t_n-\lfloor s \rfloor)} ds
%=&  \int_{0}^{t_n}  e^{-\lambda_j(t_n-s)}(1-e^{-\lambda_j(s-\lfloor s \rfloor)} )ds\\
\le & \lambda_j h  \int_{0}^{t_n} e^{-\lambda_j(t_n-s)}ds\le Ch,
\end{align*}
the Parseval equality leads to 
\begin{align*}
&\left\| \int_{0}^{t_n}\big( S^M_{t_n-s}-S^M_{t_n-\lfloor s \rfloor} \big)F(X^{K,M}_{t_{n-1}}) ds \right\|_V^2\\
=& \sum_{j=1}^{M} \langle   \int_{0}^{t_n}\big( S^M_{t_n-s}-S^M_{t_n-\lfloor s \rfloor} \big)F(X^{K,M}_{t_{n-1}}) ds ,e_j  \rangle _V^2\\
=& \sum_{j=1}^{M}\left(  \int_{0}^{t_n} \langle \big( S^M_{t_n-s}-S^M_{t_n-\lfloor s \rfloor} \big)F(X^{K,M}_{t_{n-1}}) ,e_j  \rangle _Vds \right)^2\\
\le & \sum_{j=1}^{\infty}\left(  \int_{0}^{t_n}  e^{-\lambda_j(t_n-s)}-e^{-\lambda_j(t_n-\lfloor s \rfloor)} ds \right)^2\langle F(X^{K,M}_{t_{n-1}}) ,e_j  \rangle _V^2\\
\le &Ch^2\left\|F(X^{K,M}_{t_{n-1}})\right\|_V^2.
\end{align*}
Then we obtain
\begin{align*}
J_{22}\le C\Big(1+\|u_0\|_{\dot{V}^{2}}\Big) h.
\end{align*}

For the third part $J_3$, it holds that
\begin{align*}
J_3\le Ch\sum_{i=0}^{n-1}\left\| X^{K,M}_{t_i}-X^{K,M,N}_{t_i}  \right\|_{L^2(\Omega;V)}.
\end{align*}

In order to deal with $J_4$, we deduce from Lemma \ref{lm-ito} and \cite[Lemma 4.8]{wang17BIT} that
\begin{align*}
J_4^2=&\sum_{i=1}^{K}\int_{0}^{t_n}\int_{0}^{t_n}\langle \big(S^M_{t_n-u}- S^M_{t_n-\lfloor u \rfloor} \big){\bf Q}^{\frac12}f_{i},\big(S^M_{t_n-v}- S^M_{t_n-\lfloor v \rfloor} \big){\bf Q}^{\frac12}f_{i}\rangle_V\phi(u,v)dudv\\
\le & C h^{2H+\gamma-1}\left\|A^{\frac{\gamma-1}{2}}\right\|^2_{\mathcal{L}^0_2}.
\end{align*}

Gathering the above estimates and using $h=\frac{T}{N}$, we conclude
	\begin{align*}
\sup_{n=0,\cdots,N}\left\| X^{K,M}_{t_n} -X^{K,M,N}_{t_{n}} \right\|_{L^2(\Omega;V)}\le C\Big(1+\|u_0\|_{\dot{V}^{2H+\gamma-1}}\Big) N^{-1}.
\end{align*}

\end{proof}

\begin{lemma}\label{J1}
	Let the assumptions be satisfied as in Theorem \ref{lm-tem}, then
	\begin{align*}
	J_1=\left\| \int_{0}^{t_n} S^M_{t_n-s}\Big(F(X^{K,M}_s)- F(X^{K,M}_{\lfloor s \rfloor}) \Big)ds \right\|_{L^2(\Omega;V)}\le C\Big(1+\|u_0\|_{\dot{V}^{2H+\gamma-1}}\Big) h.
	\end{align*}
\end{lemma}
\begin{proof}
Applying the Taylor's expansion leads to
\begin{align*}
&F(X^{K,M}_s)- F(X^{K,M}_{\lfloor s \rfloor})\\
=&F'(X^{K,M}_{\lfloor s \rfloor})\Big(X^{K,M}_{s}-X^{K,M}_{\lfloor s \rfloor}\Big)\\
&+\int_{0}^{1}\int_{0}^{1}\theta F''((1-\tilde{\theta})  X^{K,M}_{\lfloor s \rfloor} +\tilde{\theta}(\theta X^{K,M}_{s}+(1-\theta)X^{K,M}_{\lfloor s \rfloor}))\Big(X^{K,M}_{s}-X^{K,M}_{\lfloor s \rfloor}\Big)^2d\tilde{\theta}d\theta\\
=&F'(X^{K,M}_{\lfloor s \rfloor})\big(S^M_{s-\lfloor s \rfloor}-{\rm Id}_V\big)X^{K,M}_{\lfloor s \rfloor}\\
&+F'(X^{K,M}_{\lfloor s \rfloor})\int_{\lfloor s \rfloor}^{s}S^M_{s-\sigma}F(X^{K,M}_{\sigma})d\sigma
+F'(X^{K,M}_{\lfloor s \rfloor})\int_{\lfloor s \rfloor}^{s}S^M_{s-\sigma}dW^{{\bf Q},K}_\sigma\\
&+\int_{0}^{1}\int_{0}^{1}\theta F''( (1-\tilde{\theta})X^{K,M}_{\lfloor s \rfloor} +\tilde{\theta} (\theta X^{K,M}_{s}+(1-\theta)X^{K,M}_{\lfloor s \rfloor}))\Big(X^{K,M}_{s}-X^{K,M}_{\lfloor s \rfloor}\Big)^2d\tilde{\theta}d\theta.
\end{align*}
Therefore,
\begin{align*}
J_1\le&\left\| \int_{0}^{t_n} S^M_{t_n-s}  F'(X^{K,M}_{\lfloor s \rfloor})\big(S^M_{s-\lfloor s \rfloor}-{\rm Id}_V\big)X^{K,M}_{\lfloor s \rfloor}  ds \right\|_{L^2(\Omega;V)}\\
&+\left\| \int_{0}^{t_n} S^M_{t_n-s} F'(X^{K,M}_{\lfloor s \rfloor})\int_{\lfloor s \rfloor}^{s}S^M_{s-\sigma}F(X^{K,M}_{\sigma})d\sigma ds \right\|_{L^2(\Omega;V)}\\
&+\left\| \int_{0}^{t_n} S^M_{t_n-s}  F'(X^{K,M}_{\lfloor s \rfloor})\int_{\lfloor s \rfloor}^{s}S^M_{s-\sigma}dW^{{\bf Q},K}_\sigma  ds \right\|_{L^2(\Omega;V)}\\
&+\Bigg\| \int_{0}^{t_n} S^M_{t_n-s}  \int_{0}^{1}\int_{0}^{1}\theta F''((1-\tilde{\theta} ) X^{K,M}_{\lfloor s \rfloor} +\tilde{\theta}(\theta X^{K,M}_{s}+(1-\theta)X^{K,M}_{\lfloor s \rfloor}))\\
&\quad\quad\times\Big(X^{K,M}_{s}-X^{K,M}_{\lfloor s \rfloor}\Big)^2d\tilde{\theta}d\theta  ds \Bigg\|_{L^2(\Omega;V)}\\
=&:J_{11}+J_{12}+J_{13}+J_{14}.
\end{align*}
In the following, we show the estimates for the four terms separately.

By Lemma \ref{lm-A}, we have
\begin{align*}
J_{11}&\le \int_{0}^{t_n} \left\|A^{-1}\big(S^M_{s-\lfloor s \rfloor}-{\rm Id}_V\big)AX^{K,M}_{\lfloor s \rfloor} \right\|_{L^2(\Omega;V)} ds \\
&\le C h\int_{0}^{t_n} \left\|X^{K,M}_{\lfloor s \rfloor} \right\|_{L^2(\Omega;\dot{V}^{2} )} ds \\
&\le C\Big(1+\|u_0\|_{\dot{V}^{2H+\gamma-1}}\Big) h.
\end{align*}

Since Assumption \ref{F} implies that $F$ is linear growth, we get 
\begin{align*}
J_{12}&\le C \int_{0}^{t_n} \int_{\lfloor s \rfloor}^{s}\left\|F(X^{K,M}_{\sigma})\right\|_{L^2(\Omega;V)}d\sigma ds \\
&\le C\Big(1+\|u_0\|_{\dot{V}^{2H+\gamma-1}}\Big) h.
\end{align*}

Based on Theorem \ref{tm-regu4} and arguments in Lemma \ref{lm-esti-M}, we take $\frac12<r<1$ and obtain
\begin{align*}
J_{14}%&\le C  \int_{0}^{t_n} |t_n-s|^{\frac{r}{2}}\left\|\Big(X^{K,M}_{s}-X^{K,M}_{\lfloor s \rfloor}\Big)^2\right\|_{L^2(\Omega;V)}ds \\
&\le C  \int_{0}^{t_n} |t_n-s|^{-\frac{r}{2}}\left\|X^{K,M}_{s}-X^{K,M}_{\lfloor s \rfloor}\right\|^2_{L^4(\Omega;V)}ds\\
&\le C \Big(1+\|u_0\|_{\dot{V}^{2H+\gamma-1}}\Big)^2 \int_{0}^{t_n}|t_n-s|^{-\frac{r}{2}} |s-\lfloor s \rfloor|ds\\
&\le C  \Big(1+\|u_0\|_{\dot{V}^{2H+\gamma-1}}\Big)^2h.
\end{align*}

It remains to consider the third term $J_{13}$ which can be rewritten as
\begin{align*}
J^2_{13}%&=\left\|\sum_{i=0}^{n-1}\int_{t_i}^{t_{i+1}}S^M_{t_n-s}F'(X^{K,M}_{t_i})\int_{t_i}^{s}S^M_{s-\sigma}dW^{{\bf Q},K}_\sigma ds\right\|^2_{L^2{(\Omega;V)}}\\
&=\left\|\int_{0}^{t_{n}}S^M_{t_n-s}F'(X^{K,M}_{\lfloor s \rfloor})\int_{\lfloor s \rfloor}^{s}S^M_{s-\sigma}dW^{{\bf Q},K}_\sigma ds\right\|^2_{L^2{(\Omega;V)}}\\
&=\mathbb{E}\Bigg[\bigg\|\int_{0}^{t_{n}}S^M_{t_n-s}F'(X^{K,M}_{\lfloor s \rfloor})\int_{\lfloor s \rfloor}^{s}S^M_{s-\sigma}dW^{{\bf Q},K}_\sigma ds\bigg\|_V^2\Bigg].
\end{align*}
 The Parseval equality leads to 
\begin{align}
J^2_{13}
=&\mathbb{E}\Bigg[\langle \int_{0}^{t_{n}}S^M_{t_n-s_1}F'(X^{K,M}_{\lfloor s_1 \rfloor})\int_{\lfloor s_1 \rfloor}^{s_1}S^M_{s_1-\sigma_1}dW^{{\bf Q},K}_{\sigma_1} ds_1,\label{J13}\\
&\quad\int_{0}^{t_{n}}S^M_{t_n-s_2}F'(X^{K,M}_{\lfloor s_2 \rfloor})\int_{\lfloor s_2 \rfloor}^{s_2}S^M_{s_2-\sigma_2}dW^{{\bf Q},K}_{\sigma_2} ds_2\rangle_V\Bigg]\nonumber\\
=&\mathbb{E}\Bigg[\sum_{i=1}^{M}\langle \int_{0}^{t_{n}}S^M_{t_n-s_1}F'(X^{K,M}_{\lfloor s_1 \rfloor})\int_{\lfloor s_1 \rfloor}^{s_1}S^M_{s_1-\sigma_1}dW^{{\bf Q},K}_{\sigma_1} ds_1,e_i\rangle_V\nonumber\\
&\quad\times\langle \int_{0}^{t_{n}}S^M_{t_m-s_2}F'(X^{K,M}_{\lfloor s_2 \rfloor})\int_{\lfloor s_2 \rfloor}^{s_2}S^M_{s_2-\sigma_2}dW^{{\bf Q},K}_{\sigma_2} ds_2,e_i\rangle_V\Bigg]\nonumber\\
=&\mathbb{E}\Bigg[\sum_{i=1}^{M}\int_{0}^{t_{n}}\int_{0}^{t_{n}}\langle S^M_{t_n-s_1}F'(X^{K,M}_{\lfloor s_1 \rfloor})\int_{\lfloor s_1 \rfloor}^{s_1}S^M_{s_1-\sigma_1}dW^{{\bf Q},K}_{\sigma_1} ,e_i\rangle_V\nonumber\\
&\quad \times\langle S^M_{t_n-s_2}F'(X^{K,M}_{\lfloor s_2 \rfloor})\int_{\lfloor s_2 \rfloor}^{s_2}S^M_{s_2-\sigma_2}dW^{{\bf Q},K}_{\sigma_2} ,e_i\rangle_Vds_1ds_2\Bigg]\nonumber\\
=&\mathbb{E}\Bigg[\sum_{i=1}^{M}\int_{0}^{t_{n}}\int_{0}^{t_{n}}\langle F'(X^{K,M}_{\lfloor s_1 \rfloor})\int_{\lfloor s_1 \rfloor}^{s_1}S^M_{s_1-\sigma_1}dW^{{\bf Q},K}_{\sigma_1} , e^{-\lambda_i(t_n-s_1)}e_i\rangle_V\nonumber\\
&\quad \times\langle F'(X^{K,M}_{\lfloor s_2 \rfloor})\int_{\lfloor s_2 \rfloor}^{s_2}S^M_{s_2-\sigma_2}dW^{{\bf Q},K}_{\sigma_2} ,e^{-\lambda_i(t_n-s_2)}e_i\rangle_Vds_1ds_2\Bigg]\nonumber\\
=&\mathbb{E}\Bigg[\sum_{i=1}^{M}\int_{0}^{t_{n}}\int_{0}^{t_{n}}e^{-\lambda_i(t_n-s_1)}e^{-\lambda_i(t_n-s_2)}\langle F'(X^{K,M}_{\lfloor s_1 \rfloor})\int_{\lfloor s_1 \rfloor}^{s_1}S^M_{s_1-\sigma_1}dW^{{\bf Q},K}_{\sigma_1} , e_i\rangle_V\nonumber\\
&\quad\times\langle F'(X^{K,M}_{\lfloor s_2 \rfloor})\int_{\lfloor s_2 \rfloor}^{s_2}S^M_{s_2-\sigma_2}dW^{{\bf Q},K}_{\sigma_2} ,e_i\rangle_Vds_1ds_2\Bigg]\nonumber\\
=&\sum_{i=1}^{M}\int_{0}^{t_{n}}\int_{0}^{t_{n}}e^{-\lambda_i(t_n-s_1)}e^{-\lambda_i(t_n-s_2)}\sum_{j_1=1}^{K}\sum_{j_2=1}^{K}
\mathbb{E}\Bigg[\langle F'(X^{K,M}_{\lfloor s_1 \rfloor})\nonumber\\
&\times\int_{\lfloor s_1 \rfloor}^{s_1}S^M_{s_1-\sigma_1}{\bf Q}^\frac12 f_{j_1} d\beta^{j_1}_{\sigma_1} , e_i\rangle_V\langle F'(X^{K,M}_{\lfloor s_2 \rfloor})\int_{\lfloor s_2 \rfloor}^{s_2}S^M_{s_2-\sigma_2}{\bf Q}^\frac12 f_{j_2} d\beta^{j_2}_{\sigma_2} ,e_i\rangle_V\Bigg]ds_1ds_2.\nonumber
\end{align}
The definition of the inner product in terms of the Hilbert space $V$ and the stochastic Fubini theorem produce
\begin{align}
&\mathbb{E}\Bigg[
\langle F'(X^{K,M}_{\lfloor s_1 \rfloor})\int_{\lfloor s_1 \rfloor}^{s_1}S^M_{s_1-\sigma_1}{\bf Q}^\frac12 f_{j_1} d\beta^{j_1}_{\sigma_1} , e_i\rangle_V\langle F'(X^{K,M}_{\lfloor s_2 \rfloor})\int_{\lfloor s_2 \rfloor}^{s_2}S^M_{s_2-\sigma_2}{\bf Q}^\frac12 f_{j_2} d\beta^{j_2}_{\sigma_2} ,e_i\rangle_V\Bigg]\label{J13-2}\\
=&\mathbb{E}\Bigg[
\int_{0}^{1}f'(X^{K,M}_{\lfloor s_1 \rfloor}(x_1))\bigg(\int_{\lfloor s_1 \rfloor}^{s_1}S^M_{s_1-\sigma_1}{\bf Q}^\frac12 f_{j_1}d\beta^{j_1}_{\sigma_1}\bigg)(x_1) e_i(x_1) dx_1\nonumber\\
&\quad\times\int_{0}^{1}f'(X^{K,M}_{\lfloor s_2\rfloor}(x_2))\bigg(\int_{\lfloor s_2 \rfloor}^{s_2}S^M_{s_2-\sigma_2}{\bf Q}^\frac12 f_{j_2}d\beta^{j_2}_{\sigma_2}\bigg)(x_2) e_i(x_2) dx_2\Bigg]\nonumber\\
=&\mathbb{E}\Bigg[
\int_{0}^{1}f'(X^{K,M}_{\lfloor s_1 \rfloor}(x_1))\sum_{l_1=1}^{\infty}\bigg(\int_{\lfloor s_1 \rfloor}^{s_1}\langle S^M_{s_1-\sigma_1}{\bf Q}^\frac12 f_{j_1},e_{l_1}\rangle_Vd\beta^{j_1}_{\sigma_1}\bigg)e_{l_1}(x_1) e_i(x_1) dx_1\nonumber\\
&\quad\times\int_{0}^{1}f'(X^{K,M}_{\lfloor s_2\rfloor}(x_2))\sum_{l_2=1}^{\infty}\bigg(\int_{\lfloor s_2 \rfloor}^{s_2}\langle S^M_{s_2-\sigma_2}{\bf Q}^\frac12 f_{j_2},e_{l_2}\rangle_Vd\beta^{j_2}_{\sigma_2}\bigg)e_{l_2}(x_2) e_i(x_2) dx_2\Bigg]\nonumber\\
=&
\int_{0}^{1}\int_{0}^{1}\sum_{l_1=1}^{\infty}\sum_{l_2=1}^{\infty}\mathbb{E}\Bigg[f'(X^{K,M}_{\lfloor s_1 \rfloor}(x_1))\bigg(\int_{\lfloor s_1 \rfloor}^{s_1}\langle S^M_{s_1-\sigma_1}{\bf Q}^\frac12 f_{j_1},e_{l_1}\rangle_Vd\beta^{j_1}_{\sigma_1}\bigg)\nonumber\\
&\quad \times f'(X^{K,M}_{\lfloor s_2\rfloor}(x_2))\bigg(\int_{\lfloor s_2 \rfloor}^{s_2}\langle S^M_{s_2-\sigma_2}{\bf Q}^\frac12 f_{j_2},e_{l_2}\rangle_Vd\beta^{j_2}_{\sigma_2}\bigg)\Bigg]e_{l_1}(x_1) e_i(x_1) e_{l_2}(x_2) e_i(x_2) dx_1dx_2.\nonumber
\end{align}
According to Lemma \ref{lm-EIBP}, we have
\begin{align}
&\mathbb{E}\Bigg[f'(X^{K,M}_{\lfloor s_1 \rfloor}(x_1))\bigg(\int_{\lfloor s_1 \rfloor}^{s_1}\langle S^M_{s_1-\sigma_1}{\bf Q}^\frac12 f_{j_1},e_{l_1}\rangle_Vd\beta^{j_1}_{\sigma_1}\bigg)\label{J13-3}\\
&\quad \times f'(X^{K,M}_{\lfloor s_2\rfloor}(x_2))\bigg(\int_{\lfloor s_2 \rfloor}^{s_2}\langle S^M_{s_2-\sigma_2}{\bf Q}^\frac12 f_{j_2},e_{l_2}\rangle_Vd\beta^{j_2}_{\sigma_2}\bigg)\Bigg]\nonumber\\
=&\mathbb{E}\Bigg[\int_{0}^{T}\int_{\lfloor s_1 \rfloor}^{s_1}\int_{0}^{T}\int_{\lfloor s_2 \rfloor}^{s_2}\Big(D_v\Big(D_u\Big[f'(X^{K,M}_{\lfloor s_1 \rfloor}(x_1))f'(X^{K,M}_{\lfloor s_2\rfloor}(x_2)\Big]\Big)^{j_2}\Big)^{j_1}\nonumber\\
&\quad \times\langle S^M_{s_2-\sigma_2}{\bf Q}^\frac12 f_{j_2},e_{l_2}\rangle_V\phi(\sigma_2,u)d\sigma_2du
\langle S^M_{s_1-\sigma_1}{\bf Q}^\frac12 f_{j_1},e_{l_1}\rangle_V\phi(\sigma_1,v)d\sigma_1dv\Bigg]\nonumber\\
&+\mathbb{E}\Bigg[f'(X^{K,M}_{\lfloor s_1 \rfloor}(x_1))f'(X^{K,M}_{\lfloor s_2\rfloor}(x_2))\nonumber\\
&\quad \times\bigg(\int_{\lfloor s_1 \rfloor}^{s_1}
\int_{\lfloor s_2 \rfloor}^{s_2}\langle S^M_{s_1-\sigma_1}{\bf Q}^\frac12 f_{j_1},e_{l_1}\rangle_V \langle S^M_{s_2-\sigma_2}{\bf Q}^\frac12 f_{j_2},e_{l_2}\rangle_V \phi(\sigma_1,\sigma_2)d\sigma_2d\sigma_1\bigg)\Bigg]{\mathds{1}}_{j_1=j_2}.\nonumber
\end{align}
Substituting \eqref{J13-2}-\eqref{J13-3} into \eqref{J13}, we obtain 
\begin{align*}
J_{13}^2
=&\sum_{i=1}^{M}\int_{0}^{t_{n}}\int_{0}^{t_{n}}e^{-\lambda_i(t_n-s_1)}e^{-\lambda_i(t_n-s_2)}
\sum_{j_1=1}^{K}\sum_{j_2=1}^{K}
\int_{0}^{1}\int_{0}^{1}\sum_{l_1=1}^{\infty}\sum_{l_2=1}^{\infty}\nonumber\\
&\mathbb{E}\Bigg[f'(X^{K,M}_{\lfloor s_1 \rfloor}(x_1))f'(X^{K,M}_{\lfloor s_2\rfloor}(x_2))\bigg(\int_{\lfloor s_1 \rfloor}^{s_1}\langle S^M_{s_1-\sigma_1}{\bf Q}^\frac12 f_{j_1},e_{l_1}\rangle_Vd\beta^{j_1}_{\sigma_1}\bigg)\nonumber\\
&\quad \times\bigg(\int_{\lfloor s_2 \rfloor}^{s_2}\langle S^M_{s_2-\sigma_2}{\bf Q}^\frac12 f_{j_2},e_{l_2}\rangle_Vd\beta^{j_2}_{\sigma_2}\bigg)\Bigg] 
e_{l_1}(x_1) e_i(x_1) e_{l_2}(x_2) e_i(x_2) dx_1dx_2ds_1ds_2\\\nonumber
=&\sum_{i=1}^{M}\int_{0}^{t_{n}}\int_{0}^{t_{n}}e^{-\lambda_i(t_n-s_1)}e^{-\lambda_i(t_n-s_2)}\sum_{j_1=1}^{K}\sum_{j_2=1}^{K}\int_{0}^{1}\int_{0}^{1}\sum_{l_1=1}^{\infty}\sum_{l_2=1}^{\infty}\nonumber\\
&\mathbb{E}\Bigg[\int_{0}^{T}\int_{\lfloor s_1 \rfloor}^{s_1}\int_{0}^{T}\int_{\lfloor s_2 \rfloor}^{s_2}\Big(D_v\Big(D_u\Big[f'(X^{K,M}_{\lfloor s_1 \rfloor}(x_1))f'(X^{K,M}_{\lfloor s_2\rfloor}(x_2)\Big]\Big)^{j_2}\Big)^{j_1}\nonumber\\
&\quad \times\langle S^M_{s_2-\sigma_2}{\bf Q}^\frac12 f_{j_2},e_{l_2}\rangle_V\langle S^M_{s_1-\sigma_1}{\bf Q}^\frac12 f_{j_1},e_{l_1}\rangle_V\phi(\sigma_2,u)\phi(\sigma_1,v)d\sigma_2dud\sigma_1dv\Bigg]\nonumber\\
&\quad\times e_{l_1}(x_1) e_i(x_1) e_{l_2}(x_2) e_i(x_2) dx_1dx_2ds_1ds_2\\\nonumber
&+\sum_{i=1}^{M}\int_{0}^{t_{n}}\int_{0}^{t_{n}}e^{-\lambda_i(t_n-s_1)}e^{-\lambda_i(t_n-s_2)}\sum_{j=1}^{K}\int_{0}^{1}\int_{0}^{1}\sum_{l_1=1}^{\infty}\sum_{l_2=1}^{\infty}\nonumber\\
&\mathbb{E}\Bigg[f'(X^{K,M}_{\lfloor s_1 \rfloor}(x_1))f'(X^{K,M}_{\lfloor s_2\rfloor}(x_2))\bigg(\int_{\lfloor s_1 \rfloor}^{s_1}
\int_{\lfloor s_2 \rfloor}^{s_2}\langle S^M_{s_1-\sigma_1}{\bf Q}^\frac12 f_{j},e_{l_1}\rangle_V\nonumber\\
&\quad \times\langle S^M_{s_2-\sigma_2}{\bf Q}^\frac12 f_{j},e_{l_2}\rangle_V \phi(\sigma_1,\sigma_2)d\sigma_2d\sigma_1\bigg)\Bigg]e_{l_1}(x_1) e_i(x_1) e_{l_2}(x_2) e_i(x_2) dx_1dx_2ds_1ds_2\\
=&:I_1+I_2.
\end{align*}
Noting
\begin{align*}
I_2=&\sum_{i=1}^{M}\int_{0}^{t_{n}}\int_{0}^{t_{n}}e^{-\lambda_i(t_n-s_1)}e^{-\lambda_i(t_n-s_2)}\sum_{j=1}^{K}\int_{\lfloor s_1 \rfloor}^{s_1}
\int_{\lfloor s_2 \rfloor}^{s_2}\mathbb{E}\Bigg[\langle F'(X^{K,M}_{\lfloor s_1 \rfloor})S^M_{s_1-\sigma_1}{\bf Q}^\frac12 f_{j} ,e_i\rangle_V \\
&\quad\times F'(X^{K,M}_{\lfloor s_2\rfloor})S^M_{s_2-\sigma_2}{\bf Q}^\frac12 f_{j},e_i\rangle_V\Bigg] \phi(\sigma_1,\sigma_2)d\sigma_2d\sigma_1ds_1ds_2,
\end{align*}
we have
\begin{align*}
|I_2|\le C&\sum_{i=1}^{M}\int_{0}^{t_{n}}\int_{0}^{t_{n}}e^{-\lambda_i(t_n-s_1)}e^{-\lambda_i(t_n-s_2)}\sum_{j=1}^{K}\int_{\lfloor s_1 \rfloor}^{s_1}
\int_{\lfloor s_2 \rfloor}^{s_2}\big\|{\bf Q}^\frac12 f_{j}\big\|^2_V \phi(\sigma_1,\sigma_2)d\sigma_2d\sigma_1ds_1ds_2\\
\le C&\|A^{\frac{\gamma-1}{2}}\|^2_{\mathcal{L}_2^0}\sum_{i=1}^{\infty}\int_{0}^{t_{n}}\int_{0}^{t_{n}}e^{-\lambda_i(t_n-s_1)}e^{-\lambda_i(t_n-s_2)}\int_{\lfloor s_1 \rfloor}^{s_1}
\int_{\lfloor s_2 \rfloor}^{s_2} \phi(\sigma_1,\sigma_2)d\sigma_2d\sigma_1ds_1ds_2.
\end{align*}
Based on the chain rule of the Malliavin derivative, we get
\begin{align*}
&\Big(D_v\Big(D_u\Big[f'(X^{K,M}_{\lfloor s_1 \rfloor}(x_1))f'(X^{K,M}_{\lfloor s_2\rfloor}(x_2))\Big]\Big)^{j_2}\Big)^{j_1}\\
=&\Big(D_v\Big[f''(X^{K,M}_{\lfloor s_1 \rfloor}(x_1))\Big(D_uX^{K,M}_{\lfloor s_1 \rfloor}(x_1)\Big)^{j_2}f'(X^{K,M}_{\lfloor s_2\rfloor}(x_2))\\
&~~+
f'(X^{K,M}_{\lfloor s_1 \rfloor}(x_1))f''(X^{K,M}_{\lfloor s_2\rfloor}(x_2))\Big(D_uX^{K,M}_{\lfloor s_2\rfloor}(x_2)\Big)^{j_2}
\Big]\Big)^{j_1}\\
=&f'''(X^{K,M}_{\lfloor s_1 \rfloor}(x_1))\Big(D_vX^{K,M}_{\lfloor s_1 \rfloor}(x_1)\Big)^{j_1}\Big(D_uX^{K,M}_{\lfloor s_1 \rfloor}(x_1)\Big)^{j_2}f'(X^{K,M}_{\lfloor s_2\rfloor}(x_2))\\
&+f''(X^{K,M}_{\lfloor s_1 \rfloor}(x_1))\Big(D_v\Big(D_uX^{K,M}_{\lfloor s_1 \rfloor}(x_1)\Big)^{j_2}\Big)^{j_1}f'(X^{K,M}_{\lfloor s_2\rfloor}(x_2))\\
&+f''(X^{K,M}_{\lfloor s_1 \rfloor}(x_1))\Big(D_u(X^{K,M}_{\lfloor s_1 \rfloor}(x_1)\Big)^{j_2}f''(X^{K,M}_{\lfloor s_2\rfloor}(x_2))\Big(D_vX^{K,M}_{\lfloor s_2\rfloor}(x_2)\Big)^{j_1}\\
&+f''(X^{K,M}_{\lfloor s_1 \rfloor}(x_1))\Big(D_vX^{K,M}_{\lfloor s_1 \rfloor}(x_1)\Big)^{j_1}f''(X^{K,M}_{\lfloor s_2\rfloor}(x_2))\Big(D_u X^{K,M}_{\lfloor s_2\rfloor}(x_2)\Big)^{j_2}\\
&+f'(X^{K,M}_{\lfloor s_1 \rfloor}(x_1))f'''(X^{K,M}_{\lfloor s_2\rfloor}(x_2))\Big(D_vX^{K,M}_{\lfloor s_2\rfloor}(x_2)\Big)^{j_1}\Big(D_u X^{K,M}_{\lfloor s_2\rfloor}(x_2)\Big)^{j_2}\\
&+f'(X^{K,M}_{\lfloor s_1 \rfloor}(x_1))f''(X^{K,M}_{\lfloor s_2\rfloor}(x_2))\Big(D_v\Big(D_u X^{K,M}_{\lfloor s_2\rfloor}(x_2)\Big)^{j_2}\Big)^{j_1}\\
=&:I^{(1)}_{x_1,x_2}+I_{x_1,x_2}^{(2)}+I_{x_1,x_2}^{(3)}+I_{x_1,x_2}^{(4)}+I_{x_1,x_2}^{(5)}+I_{x_1,x_2}^{(6)},
\end{align*}
which implies that 
\begin{align*}
I_1=&\sum_{i=1}^{M}\int_{0}^{t_{n}}\int_{0}^{t_{n}}e^{-\lambda_i(t_n-s_1)}e^{-\lambda_i(t_n-s_2)}\sum_{j_1=1}^{K}\sum_{j_2=1}^{K}\int_{0}^{T}\int_{\lfloor s_1 \rfloor}^{s_1}\int_{0}^{T}\int_{\lfloor s_2 \rfloor}^{s_2}\\
&\mathbb{E}\Bigg[\int_{0}^{1}\int_{0}^{1}\Big(I^{(1)}_{x_1,x_2}+I_{x_1,x_2}^{(2)}+I_{x_1,x_2}^{(3)}+I_{x_1,x_2}^{(4)}+I_{x_1,x_2}^{(5)}+I_{x_1,x_2}^{(6)}\Big)\\
&\quad \times \Big(S^M_{s_2-\sigma_2}{\bf Q}^\frac12 f_{j_2} \Big)(x_2)\Big(S^M_{s_1-\sigma_1}{\bf Q}^\frac12 f_{j_1}\Big)(x_1)e_i(x_1)  e_i(x_2) dx_1dx_2\Bigg]\\
&\times  \phi(\sigma_2,u)\phi(\sigma_1,v)d\sigma_2dud\sigma_1dvds_1ds_2.
\end{align*}
Since for any $r>\frac12$ and $\varphi_1,\varphi_2\in \dot{V}^{r}$, the pointwise multiplication $\varphi_1\cdot \varphi_2$ satisfies 
\begin{align*}
\|\varphi_1\cdot \varphi_2\|_{ \dot{V}^{r}}\le \|\varphi_1\|_{ \dot{V}^{r}} \|\varphi_2\|_{ \dot{V}^{r}} ,
\end{align*}
it follows from Lemma \ref{lm-esti-M} and Assumption \ref{F} that for any $\frac12<r<2$, 
\begin{align*}
&\left|\int_{0}^{1}\int_{0}^{1}
I^{(1)}_{x_1,x_2}
\Big(S^M_{s_2-\sigma_2}{\bf Q}^\frac12 f_{j_2} \Big)(x_2)\Big(S^M_{s_1-\sigma_1}{\bf Q}^\frac12 f_{j_1}\Big)(x_1) e_i(x_1)  e_i(x_2) dx_1dx_2\right|\\
\le &C\Big\|\Big(D_vX^{K,M}_{\lfloor s_1 \rfloor}\Big)^{j_1}\Big\|_{ \dot{V}^{r}}\Big\|\Big(D_uX^{K,M}_{\lfloor s_1 \rfloor}\Big)^{j_2}\Big\|_{ \dot{V}^{r}}\Big\|S^M_{s_1-\sigma_1}{\bf Q}^\frac12 f_{j_1}\Big\|_{ \dot{V}^{r}}
\Big\|S^M_{s_2-\sigma_2}{\bf Q}^\frac12 f_{j_2}\Big\|_{ \dot{V}^{r}}\\
\le &C \big\|{\bf Q}^\frac12 f_{j_1}\big\|_{ \dot{V}^{r}}^2\big\|{\bf Q}^\frac12 f_{j_2}\big\|_{ \dot{V}^{r}}^2
\end{align*}
and 
\begin{align*}
&\left|\int_{0}^{1}\int_{0}^{1}
I^{(2)}_{x_1,x_2}
\Big(S^M_{s_2-\sigma_2}{\bf Q}^\frac12 f_{j_2} \Big)(x_2)\Big(S^M_{s_1-\sigma_1}{\bf Q}^\frac12 f_{j_1}\Big)(x_1) e_i(x_1)  e_i(x_2) dx_1dx_2\right|\\
\le &C\Big\|\Big(D_v\Big(D_uX^{K,M}_{\lfloor s_1 \rfloor}\Big)^{j_2}\Big)^{j_1}\Big\|_{V}\Big\|S^M_{s_1-\sigma_1}{\bf Q}^\frac12 f_{j_1}\Big\|_{V}
\Big\|S^M_{s_2-\sigma_2}{\bf Q}^\frac12 f_{j_2}\Big\|_{V}\\
\le &C \big\|{\bf Q}^\frac12 f_{j_1}\big\|_{ V}^2\big\|{\bf Q}^\frac12 f_{j_2}\big\|_{V}^2.
\end{align*}
Using similar arguments for $I^{(3)},I^{(4)},I^{(5)}$ and $I^{(6)}$, we obtain that for $\gamma>\frac32$, 
\begin{align*}
|I_1|%\le %C&\sum_{i=1}^{M}\int_{0}^{t_{n}}\int_{0}^{t_{n}}e^{-\lambda_i(t_n-s_1)}e^{-\lambda_i(t_n-s_2)}
%\sum_{j_1=1}^{K}\sum_{j_2=1}^{K}\\
%&\Bigg[\int_{0}^{T}\int_{\lfloor s_1 \rfloor}^{s_1}\int_{0}^{T}\int_{\lfloor s_2 \rfloor}^{s_2}\big\|{\bf Q}^\frac12 f_{j_1}\big\|_V^2\big\|{\bf Q}^\frac12 f_{j_2}\big\|_V^2
%\phi(\sigma_2,u)d\sigma_2du\phi(\sigma_1,v)d\sigma_1dv\Bigg]ds_1ds_2\\
\le &C\|A^{\frac{\gamma-1}{2}}\|^4_{\mathcal{L}_2^0}\sum_{i=1}^{\infty}\int_{0}^{t_{n}}\int_{0}^{t_{n}}e^{-\lambda_i(t_n-s_1)}e^{-\lambda_i(t_n-s_2)}\\
&\times\Bigg[\int_{0}^{T}\int_{\lfloor s_1 \rfloor}^{s_1}\int_{0}^{T}\int_{\lfloor s_2 \rfloor}^{s_2}
\phi(\sigma_2,u)d\sigma_2du\phi(\sigma_1,v)d\sigma_1dv\Bigg]ds_1ds_2.
\end{align*}

It then suffices to prove
\begin{align}\label{J13_1}
\sum_{i=1}^{\infty}&\int_{0}^{t_{n}}\int_{0}^{t_{n}}e^{-\lambda_i(t_n-s_1)}e^{-\lambda_i(t_n-s_2)}\\
&\times\Bigg[\int_{0}^{T}\int_{\lfloor s_1 \rfloor}^{s_1}\int_{0}^{T}\int_{\lfloor s_2 \rfloor}^{s_2}
\phi(\sigma_2,u)d\sigma_2du\phi(\sigma_1,v)d\sigma_1dv\Bigg]ds_1ds_2\le Ch^2\nonumber
\end{align}
and 
\begin{align}\label{J13_2}
\sum_{i=1}^{\infty}\int_{0}^{t_{n}}\int_{0}^{t_{n}}e^{-\lambda_i(t_n-s_1)}e^{-\lambda_i(t_n-s_2)}\int_{\lfloor s_1 \rfloor}^{s_1}
\int_{\lfloor s_2 \rfloor}^{s_2} \phi(\sigma_1,\sigma_2)d\sigma_2d\sigma_1ds_1ds_2\le Ch^2.
\end{align}
Since
\begin{align*}
\int_{0}^{T}\int_{\lfloor s_1 \rfloor}^{s_1}|\sigma_1-v|^{2H-2}d\sigma_1dv=
\int_{\lfloor s_1 \rfloor}^{s_1}\int_{0}^{T}|\sigma_1-v|^{2H-2}dvd\sigma_1
\le C(H,T)h,
\end{align*}
we obtain \eqref{J13_1} by
\begin{align*}
\int_{0}^{t_n}e^{-\lambda_i(t_n-s_1)}ds_1=\frac{1-e^{-\lambda_i t_n}}{\lambda_i}\le \frac{1}{\lambda_i}=\frac{1}{i^2\pi^2}.
\end{align*}
Defining $\lceil t \rceil:=t_{i+1}$, for $t\in(t_i,t_{i+1}]$, we get 
\begin{align*}
& \int_{0}^{t_n}\int_{0}^{t_n}\int_{\lfloor s_1 \rfloor}^{s_1}
\int_{\lfloor s_2 \rfloor}^{s_2}e^{-\lambda_i(t_n-s_1)}e^{-\lambda_i(t_n-s_2)}
|\sigma_1-\sigma_2|^{2H-2}d\sigma_2d\sigma_1ds_1ds_2\\
= & \int_{0}^{t_n}\int_{0}^{t_n}\int_{\sigma_1}^{\lceil \sigma_1\rceil}
\int_{\sigma_2}^{\lceil \sigma_2\rceil}e^{-\lambda_i(t_n-s_1)}e^{-\lambda_i(t_n-s_2)}
|\sigma_1-\sigma_2|^{2H-2}ds_2ds_1d\sigma_2d\sigma_1\\
= & \int_{0}^{t_n}\int_{\sigma_1}^{\lceil \sigma_1\rceil}e^{-\lambda_i(t_n-s_1)}
\Bigg[\int_{0}^{t_n}\int_{\sigma_2}^{\lceil \sigma_2\rceil}e^{-\lambda_i(t_n-s_2)}
|\sigma_1-\sigma_2|^{2H-2}ds_2d\sigma_2\Bigg]ds_1d\sigma_1\\
= & \int_{0}^{t_n}\int_{\sigma_1}^{\lceil \sigma_1\rceil}e^{-\lambda_i(t_n-s_1)}
\Bigg[\int_{0}^{t_n}
|\sigma_1-\sigma_2|^{2H-2}\Bigg[\int_{\sigma_2}^{\lceil \sigma_2\rceil}e^{-\lambda_i(t_n-s_2)}ds_2\Bigg]d\sigma_2\Bigg]ds_1d\sigma_1\\
\le& C(H,T)h\int_{0}^{t_n}e^{-\lambda_i(t_n-s_1)}\bigg[\int_{\lfloor s_1\rfloor}^{s_1}d\sigma_1\bigg]ds_1
\le  C\frac{h^2}{i^2\pi^2},
\end{align*}
which is summable with respect to $i$ from $1$ to infinity, and then \eqref{J13_2} holds.

%Due to $H\in(\frac12,1)$, we can choose $p>1$ such that $(2H-2)p>-1$. Using the H\"older inequality with $\frac{1}{p}+\frac{1}{q}=1$, we get
%\begin{align*}
%&\int_{0}^{t_n}
%|\sigma_1-\sigma_2|^{2H-2}\Bigg[\int_{\sigma_2}^{\lceil \sigma_2\rceil}e^{-\lambda_i(t_n-s_2)}ds_2\Bigg]d\sigma_2\\
%\le & 
%\left(\int_{0}^{t_n}
%|\sigma_1-\sigma_2|^{(2H-2)p}d\sigma_2\right)^{\frac1p}
%\left(\int_{0}^{t_n}\Bigg[\int_{\sigma_2}^{\lceil \sigma_2\rceil}e^{-\lambda_i(t_n-s_2)}ds_2\Bigg]^qd\sigma_2\right)^{\frac1q}\\
%\le & 
%C(H,T,p)
%\left(\int_{0}^{t_n}\Bigg[\int_{\sigma_2}^{\lceil \sigma_2\rceil}e^{-\lambda_i(t_n-s_2)}ds_2\Bigg]^{q-1}\Bigg[\int_{\sigma_2}^{\lceil \sigma_2\rceil}e^{-\lambda_i(t_n-s_2)}ds_2\Bigg]d\sigma_2\right)^{\frac1q}\\
%\le & 
%C(H,T,p)h^{\frac{q-1}{q}}
%\left(\int_{0}^{t_n}\int_{\sigma_2}^{\lceil \sigma_2\rceil}e^{-\lambda_i(t_n-s_2)}ds_2d\sigma_2\right)^{\frac1q}.
%\end{align*}
%Combining with 
%\begin{align*}
%\int_{0}^{t_n}\int_{\sigma_2}^{\lceil \sigma_2\rceil}e^{-\lambda_i(t_n-s_2)}ds_2d\sigma_2
%=\int_{0}^{t_n}e^{-\lambda_i(t_n-s_2)}\bigg[\int_{\lfloor s_2\rfloor}^{s_2}d\sigma_2\bigg]ds_2
%\le\frac{h}{\lambda_i}=\frac{h}{i^2\pi^2}, 
%\end{align*}
%we obtain
%\begin{align*}
%& \int_{0}^{t_n}\int_{0}^{t_n}\int_{\lfloor s_1 \rfloor}^{s_1}
%\int_{\lfloor s_2 \rfloor}^{s_2}e^{-\lambda_i(t_n-s_1)}e^{-\lambda_i(t_n-s_2)}
%|\sigma_1-\sigma_2|^{2H-2}d\sigma_2d\sigma_1ds_1ds_2\\
%\le &C(H,T,p)\frac{h^2}{i^{2(1+\frac1q)}\pi^{2(1+\frac1q)}},
%\end{align*}
%which is summable with respect to $i$ from $1$ to infinity and  implies \eqref{J13_2}.
Collecting the above estimates finishes the proof.
\end{proof}

\begin{remark}
	If one uses the temporal H\"older continuity of $X^{K,M}$ to estimate $J_1$ directly, then the convergence order will be restricted by $H$.
\end{remark}

\begin{proof}[Proof of Theorem \ref{main}]
Combining Theorems \ref{tm-3}, \ref{tm-4} and \ref{tm-6}, we obtain the conclusion of Theorem \ref{main}, which is the main result of this paper.

\end{proof}

\section{Concluding remarks}\label{sec6}
In this paper, we present the strong convergence order of the exponential integrator for the SHE driven by an infinite dimensional fractional Brownian motion with Hurst parameter $H\in(\frac12,1)$ is $1$ under Assumptions \ref{x0}-\ref{Q} with $\gamma>\max\{3-2H,\frac32\}$, which establishes the first super-convergence result in temporal direction on full discretizations for SPDEs driven by infinite dimensional fractional Brownian motions with Hurst parameter $H\in(\frac12,1)$. The main idea is to utilize Malliavin calculus to estimate the stochastic integral with respect to fBm.
% and our approach is also valid for other types of numerical schemes for temporal discretization, such as the linear implicit Euler method.

Noticing that the optimal strong order $H$ of accuracy in temporal direction is achieved as long as $\gamma\ge1$, we conjecture that the strong convergence order of the exponential integrator when $1<\gamma\le \max\{3-2H,\frac32\}$ is between $H$ and $1$. Our further work is to investigate the concrete relationship between $\gamma$ and the strong convergence order in the case of $1<\gamma\le \max\{3-2H,\frac32\}$. Due to the lack of the Burkholder--Davis--Gundy inequality for SPDEs driven by fBms, more efforts should be paid to develop new techniques to deal with the associated stochastic integrals. 

For rougher case $H\in(0,\frac12)$, since the kernel of the covariance of fBm is singular, the rough path theory needs to be employed in numerical analysis.
As for general multiplicative noises, the estimates for Malliavin derivatives of the exact solution are more complicated and the strong convergence order of the full discretization is still unsolved for SPDEs driven by infinite dimensional fBms. We will leave these topics as future works.

\bibliographystyle{plain}
\bibliography{bib}

\end{document}